\documentclass[a4paper,11pt]{amsart}
\usepackage{amsmath,amssymb,amsthm}
\usepackage[all]{xy}
\usepackage{bm}
\usepackage{mathrsfs}
\usepackage{enumerate}
\usepackage{amsrefs}

\newcommand{\Spec}{\operatorname{Spec}}

\newcommand{\Hom}{\operatorname{Hom}}

\newcommand{\Ker}{\operatorname{Ker}}

\theoremstyle{definition}
\newtheorem{dfn}{Definition}[section]
\theoremstyle{plain}
\newtheorem{thm}[dfn]{Theorem}
\newtheorem{prop}[dfn]{Proposition}
\newtheorem{lem}[dfn]{Lemma}
\newtheorem{cor}[dfn]{Corollary}
\newtheorem{conj}[dfn]{Conjecture}
\theoremstyle{remark}
\newtheorem{rem}[dfn]{Remark}

\newtheorem{claim}[dfn]{Claim}
\begin{document}
\title[Rational curves on del Pezzo manifolds]{The irreducibility of the spaces of rational curves on del Pezzo manifolds}
\author[Fumiya Okamura]{Fumiya Okamura}
\address{Graduate School of Mathematics, Nagoya University, Furo-cho, Chikusa-ku, Nagoya, 464-8601, Japan}
\email{fumiya.okamura.e6@math.nagoya-u.ac.jp}
\begin{abstract}
We prove the irreducibility of the spaces of rational curves on del Pezzo manifolds of Picard rank $1$ and dimension $n \ge 4$ by analyzing the fibers of evaluation maps.
As a corollary, we prove Geometric Manin's Conjecture in these cases.
\end{abstract}
\maketitle
\section{Introduction}
Throughout this paper, we work over an algebraically closed field $k$ of characteristic $0$.
A smooth projective variety $X$ is called Fano if the anticanonical divisor $-K_{X}$ is ample.
Mori devised the bend and break method and proved the following theorem about rational curves on Fano manifolds:
\begin{thm}[\cite{Mori1982}]\label{Mori}
Let $X$ be a Fano manifold of dimension $n$.
Then for any point $p \in X$, there is a rational curve $C$ on $X$ such that 
$p \in C$ and $-K_{X} \cdot C \le n+1$.
\end{thm}
Hence it is natural to study the space of rational curves on Fano manifolds.
In particular, we are interested in the irreducibility of the spaces $\Hom (\mathbb{P}^{1}, X, d)$ of rational curves of degree $d \ge 1$ and their dimensions.
\par
Fano manifolds are classified by their indexes.
For a Fano manifold $X$, the index of $X$ is the largest integer $r$ such that $-K_{X} = rH$ for some ample divisor $H$ on $X$.
Such a divisor $H$ is called the fundamental divisor on $X$.
The index of a Fano manifold of dimension $n$ is at most $n + 1$ (e.g., \cite[Corollary 2.1.13]{Iskovskikh1999}).
Moreover, if the index of $X$ is $n+1$, then $X \cong \mathbb{P}^{n}$, and 
if the index of $X$ is $n$, then $X$ is isomorphic to a smooth quadric $Q^{n}$ in $\mathbb{P}^{n+1}$.
When $X = \mathbb{P}^{n}$ (resp. $Q^{n}$), the space $\mathrm{Hom}(\mathbb{P}^{1}, X, d)$ is irreducible of dimension $(n+1)d+n$ (resp. $nd+n$) for each $d\in \mathbb{Z}_{>0}$ (e.g., since $X$ is homogeneous and toric, it holds by the papers cited in the second paragraph from the bottom of Introduction).
So we will consider del Pezzo manifolds, which are Fano manifolds such that $-K_{X} = (n - 1)H$ for some ample divisor $H$.
Del Pezzo manifolds are completely classified (e.g., \cite[Theorem 3.3.1]{Iskovskikh1999}).
In particular, $X$ is a del Pezzo manifold of Picard rank $1$ if and only if $1\le H^{n}\le 5$, and in this case $X$ is one of the following:
\begin{enumerate}[(1)]
\item When $H^{n} = 1$, $X$ is a smooth sextic in $\mathbb{P}(1^{n},2,3)$,
\item When $H^{n} = 2$, $X$ is a smooth quartic in $\mathbb{P}(1^{n+1},2)$,
\item When $H^{n} = 3$, $X$ is a smooth cubic in $\mathbb{P}^{n + 1}$,
\item When $H^{n} = 4$, $X$ is a smooth complete intersection of two quadrics in $\mathbb{P}^{n + 2}$, 
\item When $H^{n} = 5$, $X$ is a smooth linear section of the Grassmannian $\mathrm{Gr}(2,5) \subset \mathbb{P}^{9}$. 
\end{enumerate}
In this paper, we study the spaces of rational curves on del Pezzo manifolds of dimension $n \ge 4$.
Our aim is to prove the following theorem:
\begin{thm}\label{main}
Let $X$ be a del Pezzo manifold of Picard rank $1$ and dimension $n \ge 4$ with an ample generator $H$. 
For any integer $d \ge 1$, the Kontsevich space $\overline{M}_{0,0}(X,d)$ parametrizing rational curves of $H$-degree $d$ is irreducible of the expected dimension $(n-1)d + n - 3$.
\end{thm}
For the definition of the Kontsevich spaces, see for example \cite{Fulton1997}.
For the proof of this theorem, we mainly follow the methods in \cite{Coskun2009}, \cite{Lehmann2019}.
In fact, \cite{Coskun2009} proved Theorem \ref{main} for any smooth cubic hypersurface of dimension $\ge 4$, and by a similar argument, it is easily proved the theorem when $H^{n} \ge 4$.
Our main ingredient is to analyze the fibers of the evaluation map $\mathrm{ev_{1}}:\overline{M}_{0,1}(X, 1) \rightarrow X$ for the case when $H^{n} \le 2$.
One of the main difficulties for these cases is that the fundamental divisor is not very ample. 
Our idea is to study the subvarieties which are covered by lines through a fixed point instead of the fibers of $\mathrm{ev}_{1}$.
For the case of degree $2$, we conduct a little more precise study:
For a point contained in the ramification locus of the double cover $f \colon X \rightarrow \mathbb{P}^{n}$, we give a necessary and sufficient condition for the fiber of $\mathrm{ev}_{1}$ to have the expected dimension. 
Using the above result, we can run the induction on the degree $d$ as in the above references. 
The key lemma for the induction step is the movable bend and break, which asserts that any free curve on $X$ deforms to a chain of free lines.
Note that we can focus on the study of lines thanks to the inequality in Theorem \ref{Mori} since $(n-1)H\cdot C \le n+1$ and $n\ge 4$ yield that $H\cdot C = 1$.
\par
The study of the irreducibility of $\Hom (\mathbb{P}^{1}, X, d)$ and their dimensions is motivated by Geometric Manin's Conjecture.
This conjecture predicts the growth rate of the number of components and their dimension as the degree $d$ grows.
For the precise formulation of the asymptotic formula, we need to remove the exceptional subsets.
\cite{Lehmann2019} conjectured that this exceptional set is defined by two birational invariants $a$ and $b$.
These invariants are defined in Section 2.
Subvarieties with large $a$ and $b$ invariants would contain more rational curves than we expect.
However, we prove that any del Pezzo manifold has no subvariety with large $a$ invariant.
Hence the situation is rather simple.
For details about Geometric Manin's Conjecture, see \cite{Tanimoto2021}.
\par 
The results about the number of components and their dimensions are known in various cases: 
smooth Fano hypersurfaces of certain degree ranges in  \cite{Beheshti2013}, \cite{Browning2017}, \cite{Coskun2009}, \cite{Harris2004}, \cite{Riedl2019}, 
homogeneous varieties in \cite{Kim2001}, \cite{Thomsen1998}, 
Toric varieties in \cite{Bourqui2012}, \cite{Bourqui2016}, 
del Pezzo surfaces in \cite{Beheshti2023}, \cite{Testa2005}, \cite{Testa2009}, 
smooth Fano threefolds in \cite{Beheshti2022}, \cite{Burke2022}, \cite{Lehmann2019}, \cite{Lehmann2021a}, \cite{Shimizu2022}, 
del Pezzo fibrations in \cite{Lehmann2022a}, \cite{Lehmann2023}, and so on.
\par
The paper is organized as follows.
In Section 2, we give the definitions of two invariants $a$ and $b$, and prove several properties about $a$-invariants on del Pezzo manifolds.
Section 3 is the main part of this paper, we study the spaces of lines on del Pezzo manifolds passing through a fixed point.
We mainly focus on del Pezzo manifolds of degree $\le 2$.
In Section 4, we run the induction step and prove the main theorem.
The arguments are automatic by \cite{Coskun2009} and \cite{Lehmann2019}.
In Section 5,  as a corollary of Theorem \ref{main}, we prove Geometric Manin's Conjecture for our case.
\section{$a$-, $b$-invariants}
In this section, we will introduce the $a$-invariant and the $b$-invariant.
Both invariants play an essential role to count the number of components of the scheme $\Hom (\mathbb{P}^{1}, X)$ parametrizing morphisms from $\mathbb{P}^{1}$ to a smooth projective variety $X$. 
For details about $\Hom (\mathbb{P}^{1}, X)$, see \cite[I.1]{Kollar1996}.
\par
We let $N^{1}(X)$ be the set of $\mathbb{R}$-Cartier divisors modulo numerical equivalence.
Let $\overline{\mathrm{Eff}}^{1}(X)$ be the set of pseudo-effective $\mathbb{R}$-divisors on $X$, which are elements in the closure of the set of classes of effective $\mathbb{R}$-divisors, and 
let $\mathrm{Nef}^{1}(X)$ be the set of classes of nef $\mathbb{R}$-divisors on $X$.
Similarly, let $N_{1}(X)$ be the set of $\mathbb{R}$-$1$-cycles modulo numerical equivalence.
Let $\overline{\mathrm{Eff}}_{1}(X)$ be the set of pseudo-effective $\mathbb{R}$-$1$-cycles on $X$, and 
let $\mathrm{Nef}_{1}(X)$ be the set of classes of nef $\mathbb{R}$-$1$-cycles on $X$. 
\begin{dfn}[\cite{Lehmann2019} Definition 3.1]
Let $X$ be a projective manifold and $L$ be a nef and big $\mathbb{Q}$-divisor on $X$.
Then the $a$-invariant, or the Fujita invariant, $a(X, L)$ is defined by
\[a(X,L) = \inf \{t \in \mathbb{R} \mid K_{X} + tL \in \overline{\mathrm{Eff}}^1 (X)\}.\]
This is a birational invariant by \cite[Proposition 2.7]{Hassett2015}.
Hence we define $a(X, L)$ for a singular variety $X$ by
\[a(X, L) = a(\tilde{X}, \phi^{*}L),\]
where $\phi \colon \tilde{X} \rightarrow X$ is a smooth resolution.
\end{dfn}
By \cite{Boucksom2013}, $a(X, L) > 0$ if and only if $X$ is uniruled.
\begin{prop}[\cite{Lehmann2019} Proposition 4.2]\label{expected}
Let $X$ be a projective uniruled manifold and let $L$ be a nef and big divisor on $X$.
For $\alpha \in \overline{\mathrm{Eff}}_{1}(X)$ vanishing against $K_{X} + a(X, L)L$, we take a component $M \subset \Hom (\mathbb{P}^{1}, X)$ parametrizing morphisms $f$ such that $f_{*}\mathbb{P}^{1} = \alpha$.
Consider the evaluation map 
\[\mathrm{ev} \colon \mathbb{P}^{1} \times M \rightarrow X. \]
If $\mathrm{ev}$ is not dominant, then the closure $Y$ of the image of $\mathrm{ev}$ has an $a$-value $a(Y, L) > a(X, L)$.
\end{prop}
\begin{dfn}[\cite{Lehmann2019} Definition 3.4]
Let $X$ be a projective uniruled manifold and let $L$ be a nef and big $\mathbb{Q}$-divisor.
Then a generically finite dominant morphism $f\colon Y \rightarrow X$ of degree $\ge 2$ from a projective variety $Y$ is an $a$-cover if $a(Y,f^{*}L) = a(X, L)$.
\end{dfn}
\begin{rem}
Let $X$ be a del Pezzo manifold of dimension $n$ and let $H$ be the fundamental divisor.
Then, by definition, we have $a(X, H) = n-1$.
\end{rem}
On del Pezzo manifolds, $a$-invariants behave well:
\begin{lem}\label{largea}
Let $X$ be a del Pezzo manifold of Picard rank $1$ and dimension $n \ge 4$ with an ample generator $H$. Then $X$ has no subvarieties $Y$ such that $a(Y, H) > a(X, H)$.
\end{lem}
\begin{proof}
If the codimension of a subvariety $Y$ is greater than $1$, then $a(Y, H) \le n-1 = a(X, H)$ by \cite[Lemma 3.16]{Lehmann2019a}. 
Hence we may assume that $Y$ has codimension $1$.
By \cite[Theorem 3.15]{Lehmann2019a}, for general $n-3$ hyperplanes $H_{1}, \dots , H_{n-3}$, the section $X' := X \cap H_{1}\cap \dots \cap H_{n-3}$ is a smooth del Pezzo threefold. 
If there exists a subvariety $Y$ such that $a(Y, H) > a(X, H)$, 
by \cite[Theorem 3.15]{Lehmann2019a} again, we obtain a subvariety $Y' := Y \cap H_{1}\cap \dots \cap H_{n-3} \subset X'$ such that $a(Y', H) > a(X', H)$.
However, del Pezzo threefolds have no such subvarieties by \cite[\S 6.3]{Lehmann2018}.
\end{proof}
\begin{cor}
Let $X$ be a del Pezzo manifold of Picard rank $1$ and dimension $n \ge 4$.
For any $d \ge 1$ and any component $M \subset \Hom (\mathbb{P}^{1}, X, d)$, 
the evaluation map $\mathbb{P}^{1} \times M \rightarrow X$ is dominant.
\end{cor}
The second invariant $b$ is not so important for del Pezzo manifolds.
However, we need to define it in order to formulate Geometric Manin's Conjecture.
\begin{dfn}[\cite{Lehmann2019} Definition 3.2]
Let $X$ be a projective uniruled manifold.
Let $L$ be a nef and big $\mathbb{Q}$-divisor on $X$.
We define the $b$-invariant by
\[b(X, L) = \dim F(X, L), \]
where $F(X, L) = \{\alpha \in \mathrm{Nef}_{1}(X) \mid (K_{X} + a(X, L)L) \cdot \alpha = 0\}$. 
This is also a birational invariant by \cite[Proposition 2.10]{Hassett2015}.
Thus we define $b(X, L)$ for a singular variety $X$ by
\[b(X, L) = b(\tilde{X}, \phi^{*}L),\]
where $\phi \colon \tilde{X} \rightarrow X$ is a smooth resolution.
\end{dfn}
\begin{rem}
The two invariants $a, b$ appeared in Manin's Conjecture.
For the notions of adelic line bundles and the height functions, see for example \cite{Chen2020}.
Let $X$ be a smooth Fano variety over a number field $F$ and $\mathcal{L} = (L, \| \cdot \|)$ be a nef and big line bundle with an adelic metrization. 
For each integer $T \ge 1$ and a subset $Q \subset X(F)$ , define the counting function by
\[N(Q, \mathcal{L}, T) = \#\{x \in Q \mid H_{\mathcal{L}}(x) \le T\},\]
where $H_{\mathcal{L}}$ is the height function associated to $\mathcal{L}$.
Then Manin's Conjecture asserts that there is a thin set $Z \subset X(F)$ such that 
\[N(X(F) \setminus Z, \mathcal{L}, T) \sim cT^{a(X, L)}(\log T)^{b(X, L) - 1},\quad \mbox{as}\quad T \rightarrow \infty,\]
where $c = c(F, X(F)\setminus Z, \mathcal{L})$ is Peyre's constant introduced in \cite{Batyrev1998}, \cite{Peyre1995}.
The references \cite{Batyrev1990}, \cite{Batyrev1998}, \cite{Franke1989}, \cite{Peyre1995} contributed significantly to the formulation. 
The thin set $Z \subset X(F)$ was described conjecturally by using the invariants $a, b$ in \cite{Lehmann2022}, \cite{Lehmann2017}. 
The reference \cite{Lehmann2019} was inspired by them and stated a version of Manin's conjecture for rational curves.
\end{rem}
\begin{dfn}[\cite{Lehmann2019} Definition 3.5]
Let $X$ be a projective uniruled manifold and let $L$ be a nef and big $\mathbb{Q}$-divisor. 
Then an $a$-cover $f \colon Y \rightarrow X$ is a face contracting morphism if the induced map $f_{*}\colon F(Y, f^{*}L) \rightarrow F(X, L)$ is not injective.
\end{dfn}
\section{Lines through a fixed point}
We recall the notion of free rational curves, which will be frequently used in the remaining sections. See \cite[II.3]{Kollar1996} for several results on free curves. 
\begin{dfn}
    Let $X$ be a smooth projective variety. 
    We say that a nonconstant morphism $f\colon \mathbb{P}^{1}\rightarrow X$ is free if $H^{1}(\mathbb{P}^{1}, f^{*}T_{X}(-1)) = 0$.
\end{dfn}
In this section, we prove the base case of Theorem \ref{main}.
For proving the induction step, we need to study fibers of the evaluation map of the space of lines:
\begin{thm}\label{line}
Let $X$ be a del Pezzo manifold of Picard rank $1$ and dimension $n \ge 4$ with an ample generator $H$. 
Let $\mathrm{ev}\colon \overline{M}_{0, 1}(X, 1) \rightarrow X$ be the evaluation map.
Then 
\begin{enumerate}[(1)]
\item 
A general fiber of $\mathrm{ev}$ is irreducible. 
\item 
There is a finite subset $S$ of $X$ such that
  \begin{itemize}
  \item If $p \notin S$, then $\dim (\mathrm{ev}^{-1}(p)) = n-3$.
  \item If $p \in S$, then $\dim (\mathrm{ev}^{-1}(p)) \le n-2$.
  \end{itemize}
\end{enumerate}
In particular, $\overline{M}_{0, 0}(X, 1)$ is irreducible of dimension $2n-4$.
\end{thm}
\begin{rem}
For some dimension range, one can show (1) easily by an analysis of $a$-covers.
Indeed, the generic fiber of $\mathrm{ev}$ is smooth by \cite[II.3.11 Theorem]{Kollar1996}.
By \cite[Proposition 5.15]{Lehmann2019}, the finite part of the Stein factorization of $\mathrm{ev}$ is an $a$-cover.
However, by \cite[Theorem 11.1]{Lehmann2019a}, there is no generically finite morphism with $a(Y,-f^{*}K_{X}) \ge a(X, -K_{X})$ if $X$ satisfies one of the following:
\begin{itemize}
\item $H^{n} = 2,3$, $\dim X \ge 4$, and $X$ is general in its moduli, 
\item $H^{n} = 4$ and $\dim X \ge 5$, or
\item $H^{n} = 5$ and $\dim X = 6$, i.e., $X \cong \mathrm{Gr}(2, 5)$.
\end{itemize}
Hence, in these cases, $\mathrm{ev}$ has connected fibers.
\end{rem}
Before proving Theorem \ref{line}, we introduce several lemmas.
\begin{lem}\label{spanned}
Let $X$ be a del Pezzo manifold of Picard rank $1$ and dimension $n \ge 4$.
Let $W \subset \overline{M}_{0,0}(X, 1)$ be an $m$-dimensional subvariety parametrizing lines passing through a fixed point.
Then the members of $W$ cover a subvariety of $X$ of dimension $m+1$.
\end{lem}
\begin{proof}
Let $Z \subset X$ be the subvariety which is covered by lines parametrized by $W$.
Suppose that $Z$ has dimension less than $m+1$.
Let $W' \subset \overline{M}_{0,1}(X,1)$ be the $(m+1)$-dimensional subvariety corresponding to $W$.
Consider the evaluation map $\mathrm{ev}\colon \overline{M}_{0,1}(X, 1) \rightarrow X$.
Then any fiber of the restriction morphism $\mathrm{ev}|_{W'}$ has dimension at least $1$.
In particular, this means that there is a one parameter family $T$ of lines passing through two fixed points.
By bend and break lemma (e.g., \cite[II.5.5 Corollary]{Kollar1996}), $T$ contains a point which corresponds to a non-integral $1$-cycle.
However, it is a contradiction since $T$ parametrizes lines on $X$.
\end{proof}
\begin{lem}\label{unexpected}
Let $X$ be a del Pezzo manifold of Picard rank $1$ and dimension $n \ge 4$ and let 
$\mathrm{ev} \colon \overline{M}_{0,1}(X, 1) \rightarrow X$ be the evaluation map.
Then $\dim \mathrm{ev}^{-1}(p) \le n - 2$ for any point $p \in X$.
\end{lem}
\begin{proof}
Suppose that there is a point $p \in X$ such that $\dim \mathrm{ev}^{-1}(p) \ge n - 1$.
Let $W$ be the image of an $(n-1)$-dimensional component of $\mathrm{ev}^{-1}(p)$ under the morphism $\overline{M}_{0,1}(X,1)\rightarrow \overline{M}_{0,0}(X,1)$.
Then $W$ has also dimension $n-1$.
If $W$ contains a point corresponding to a free line, then $W$ must have the expected dimension $n-3$ by \cite[II.3.5 Corollary]{Kollar1996}.
Thus $W$ parametrizes non-free lines.
On the other hand, by Lemma \ref{spanned}, $X$ is covered by lines parametrized by $W$.
However, it is a contradiction since non-free lines are contained in a closed subset of $X$ by \cite[II.3.11 Theorem]{Kollar1996}.
\end{proof}
\begin{lem}\label{sectionirr}
Let $n \ge 4$ be an integer.
If $X \subset \mathbb{P}^n$ is a smooth hypersurface,
then any hyperplane section of $X$ is reduced and irreducible.
\end{lem}
\begin{proof}
Assume that the section of $X$ cut by the hyperplane $(x_{n} = 0)$ is not irreducible.
Then there is a polynomials $F, G, H$ such that $X$ is given by the equation 
\[x_{n}F + GH = 0.\]
Since $n \ge 4$, the set $S = V(x_{n}, F, G, H)$ is nonempty.
Then, however, $X$ is singular at any point in $S$, which is a contradiction.
\end{proof}
Firstly, we prove Theorem \ref{line} for the case $H^{n} \ge 3$.
In particular, when $H^{n} = 5$, $X$ is a smooth linear section of the Grassmannian $\mathrm{Gr}(2,5)\subset \mathbb{P}^{9}$.
The papers such as \cite{Abdelkerim2012}, \cite{Chung2018}, \cite{Harris2004}, \cite{Piontkowski1999}, \cite{Takagi2012} will be of great help to study Grassmannians.
\begin{proof}[Proof of Theorem \ref{line}: $H^{n}\ge 3$ case]
    When $H^n = 3$, $X\subset \mathbb{P}^{n+1}$ is a smooth cubic hypersurface.
    In this case, the theorem is proved in \cite{Coskun2009}.
    \par
    When $H^n = 4$, $X \subset \mathbb{P}^{n+2}$ is a smooth complete intersection of two quadrics.
    We will show in a similar way to \cite{Coskun2009} that $\mathrm{ev}$ is flat with irreducible general fibers.
    Let $p = (1:0:\dots:0) \in X$ and we may assume that $T_{p}X = V(x_{1}, x_{2})$.
    Then $X$ is given by two polynomials 
    \[
    \left\{
    \begin{aligned}
    x_{0}x_{1} + F(x_{1}, \dots, x_{n+2}) &= 0\\
    x_{0}x_{2} + G(x_{1}, \dots, x_{n+2}) &= 0,
    \end{aligned}
    \right. 
    \]
    where $F,G$ are homogeneous polynomials of degree $2$.
    Any line $\ell$ on $\mathbb{P}^{n+2}$ passing through $p$ is written as 
    \[\ell = \{(s:a_{1}t:\dots:a_{n+2}t) \mid (s:t)\in \mathbb{P}^{1}\} \]
    for $(a_{1}:\dots:a_{n+2}) \in \mathbb{P}^{n+1}$.
    Hence $\ell$ is contained in $X$ if and only if 
    \[
    \left\{
    \begin{aligned}
    a_{1} st + F(a_{1}, \dots, a_{n+2})t^2 &= 0\\
    a_{2} st + G(a_{1}, \dots, a_{n+2})t^2 &= 0
    \end{aligned}
    \right. 
    \]
    for any $(s:t)\in \mathbb{P}^{1}$.
    Thus we obtain that $\mathrm{ev}^{-1}(p) \cong V(x_{1}, x_{2}, F, G) \subset \mathbb{P}^{n+1}$, which is in particular, connected.
    Moreover, if $\dim \mathrm{ev}^{-1}(p) = n-2$, then either of $F(0,0,x_{3},\dots, x_{n+2})$ or $G(0,0,x_{3},\dots, x_{n+2})$ is identically zero, or $F(0,0,x_{3},\dots, x_{n+2})$ and $G(0,0,x_{3},\dots, x_{n+2})$ have a common component.
    In any case, $X\cap V(x_{1}, x_{2})$ has dimension $n-1$ and degree at most $2$.
    However, it contradicts the Lefschetz theorem (e.g., \cite[Example 3.1.31]{Lazarsfeld2004}).
    Thus $\mathrm{ev}$ is flat with connected fibers.
    Finally, by \cite[II.3.11 Theorem]{Kollar1996}, the general fiber of $\mathrm{ev}$ is smooth.
    Therefore, the claim holds when $H^{n}=4$.
    \par
    When $H^n = 5$,  $X$ is a smooth linear section of the Grassmannian $\mathrm{Gr}(2,5)$.
    In particular, $\dim X \le 6$.
    In this case as well, we will show that $\mathrm{ev}$ is flat with irreducible general fibers.
    The smoothness of general fibers of $\mathrm{ev}$ follows by \cite[II.3.11 Theorem]{Kollar1996}.
    Hence it is enough to show that $\mathrm{ev}$ is flat with connected fibers.
    Let $V$ be a $5$-dimensional vector space and let $\mathrm{Gr}(2, 5)$ be the Grassmannian of $2$-dimensional subspaces of $V$.
    There is an embedding, called the Pl\"{u}cker embedding $\mathrm{Gr}(2,5)\rightarrow \mathbb{P}(\bigwedge^{2}V) \cong \mathbb{P}^{9}$, which maps $[\langle v, w\rangle]$ to $[v\wedge w]$.
    By \cite{Abdelkerim2012}, the space of lines on $\mathrm{Gr}(2,5)$ is isomorphic to the flag variety \[F(1,3,5) = \{(V_{1}, V_{3}) \mid V_{1} \subset V_{3} \subset V, \dim V_{i} = i\}, \] 
    that is, the line corresponding to $(V_{1}, V_{3}) \in F(1,3,5)$ is the Schubert variety 
    \[\sigma = \{[W]\in \mathrm{Gr}(2,5) \mid V_{1} \subset W\subset V_{3}\}.\]
    Hence the space of lines passing through $[W]\in \mathrm{Gr}(2,5)$ is isomorphic to 
    \[\{(V_{1}, V_{3})\in F(1,3,5) \mid V_{1}\subset W\subset V_{3}\}, \]
    which is, moreover, isomorphic to $\mathbb{P}^{1}\times \mathbb{P}^{2}$.
    Therefore, the claim holds when $\dim X = 6$. 
    \par
    Fix a basis $\{e_{1}, \dots, e_{5}\}$ of $V$ and let $z_{ij} \, (1\le i<j\le 5)$ be the corresponding coordinates of $\mathbb{P}^{9}$.
    Take a point $[W] = [\langle e_{1}, e_{2}\rangle] \in \mathrm{Gr}(2,5)$.
    There is a bijection $\mathbb{P}^{*}(\bigwedge^{2}V)\rightarrow \{Q\in \mathrm{Mat}(V) \mid {}^{t}Q = -Q\}/k^{*}$, which maps a hyperplane $H = V(\sum_{1\le i < j\le 5}a_{ij}z_{ij})$ to the skew symmetric matrix $Q_{H} := (a_{ij})_{1\le i,j \le 5}$ (i.e., $a_{ji} = -a_{ij}$) modulo scalar multiplications.
    Then $\mathrm{Gr}(2,5)\cap H$ is singular at $[W]$ if and only if $T_{[W]}\mathrm{Gr}(2,5) \subset H$.
    The embedded tangent space $T_{[W]}\mathrm{Gr}(2,5) \subset \mathbb{P}^{9}$ is spanned by $7$ points $[e_{1}\wedge e_{2}], \dots, [e_{1}\wedge e_{5}], [e_{2}\wedge e_{3}], \dots, [e_{2}\wedge e_{5}]$ by \cite{Abdelkerim2012}.
    Hence the condition $T_{[W]}\mathrm{Gr}(2,5) \subset H$ is equivalent that $W \subset \Ker Q_{H}$ (e.g., \cite[Corollary 1.6]{Piontkowski1999}).
    Since $Q_{H}$ is skew symmetric of size $5$, it cannot be regular.
    Thus $\mathrm{Gr}(2,5)\cap H$ is smooth if and only if $\dim \Ker Q_{H} = 1$.
    \par
    Suppose that $\mathrm{Gr}(2,5) \cap H$ is smooth and contains $[W]$.
    For $((x_{1}:x_{2}), (x_{3}:x_{4}:x_{5})) \in \mathbb{P}^{1}\times \mathbb{P}^{2}$, the corresponding line on $\mathrm{Gr}(2,5)$ connecting $2$ points $[e_{1}\wedge e_{2}]$ and $[(x_{1}e_{1}+x_{2}e_{2})\wedge (x_{3}e_{3}+x_{4}e_{4}+x_{5}e_{5})]$ is contained in $\mathrm{Gr}(2,5)\cap H$ if and only if 
    \[{}^{t}(x_{1}e_{1}+x_{2}e_{2})Q_{H}(x_{3}e_{3}+x_{4}e_{4}+x_{5}e_{5}) = 0.\]
    Hence the space of lines on $\mathrm{Gr}(2,5) \cap H$ is isomorphic to an ample divisor 
    \[A_{H} := V\left(\sum_{i=1}^{2} \sum_{j=3}^{5}a_{ij}x_{i}x_{j}\right) \subset \mathbb{P}^{1}\times \mathbb{P}^{2} \] 
    of bidegree $(1,1)$, where $Q_{H} = (a_{ij})_{i,j}$.
    Therefore, the claim holds when $\dim X = 5$.
    Furthermore, $A_{H}$ fails to be irreducible if and only if $\Ker Q_{H} \subset W$.
    \par
    Finally, we consider the intersection $X = \mathrm{Gr}(2,5) \cap H_{1} \cap H_{2}$ of $\mathrm{Gr}(2,5)$ and two hyperplanes $H_{1}, H_{2}$.
    Another interpretation of $X$ is the base locus of the pencil $P \subset |\mathcal{O}_{\mathrm{Gr}(2,5)}(1)|$ spanned by $H_{1}$ and $H_{2}$.
    The smoothness of $X$ is equivalent to the smoothness of $\mathrm{Gr}(2,5) \cap H_{\alpha}$ for any $H_{\alpha} \in P$ (e.g., by the Jacobian criterion).
    We may assume that $[W] = [\langle e_{1}, e_{2}\rangle] \in X$.
    Then the space of lines on $X$ passing through $[W]$ is isomorphic to the intersection of two ample divisors $A_{H_{1}} \cap A_{H_{2}} \subset \mathbb{P}^{1}\times \mathbb{P}^{2}$, or one can write as $\bigcap_{H_{\alpha}\in P}A_{H_{\alpha}}$.
    If $A_{H_{1}} = A_{H_{2}}$, then for some $H_{\alpha}$, we have $\dim \Ker Q_{H_{\alpha}} > 1$, a contradiction.
    Hence the dimension of $A_{H_{1}} \cap A_{H_{2}}$ has to be $1$ unless $\Ker Q_{H_{\alpha}} \subset W$ for any $H_{\alpha} \in P$.
    Consider the morphism $c \colon P\rightarrow \mathbb{P}(V)$ given by $c(H_{\alpha}) = \Ker Q_{H_{\alpha}}$.
    Then by \cite[Proposition 6.3]{Piontkowski1999}, $c$ is an embedding to a smooth conic.
    This implies that there exists $H_{\alpha} \in P$ such that $\Ker Q_{H_{\alpha}} \not\subset W$.
    Thus $A_{H_{1}} \cap A_{H_{2}}$ has dimension $1$, which concludes the claim when $\dim X = 4$.
\end{proof}
\subsection{Proof of Theorem \ref{line}: $H^{n} = 2$ case}
Let $X$ be a del Pezzo manifold of degree $2$.
Then $X$ is a smooth hypersurface of degree $4$ in the weighted projective space $\mathbb{P}(1^{n+1},2)$. 
Hence $X$ has the form
\[X = V(y^{2} + F), \]
where $F \in k[x_{0}, \dots, x_{n}]$ is a homogeneous polynomial of degree $4$ and $\deg y = 2$.
Then there is a double cover $f \colon X \rightarrow \mathbb{P}^{n}$ branched over a smooth quartic $B = V(F) \subset \mathbb{P}^{n}$.
Then by the projection formula, a curve $\tilde{\ell} \subset X$ is a line if and only if $f_{*} \tilde{\ell} \subset \mathbb{P}^n$ is a line.
This correspondence induces a natural transformation 
\[\overline{\mathcal{M}}_{0,0}(X,1) \rightarrow \mathit{Grass}(2, n+1), \]
of functors (notation as in \cite[\S.1]{Fulton1997} and \cite[I.1]{Kollar1996}).
Hence it defines the morphism into a Grassmannian of lines in $\mathbb{P}^{n}$
\[f_{*} \colon M := \overline{M}_{0, 0} (X, 1) \rightarrow \mathbb{G}(1, n) \cong \overline{M}_{0, 0}(\mathbb{P}^{n}, 1),\]
which is finite onto its image, say $N$.
Let $\ell \subset \mathbb{P}^n$ be a line and $\tilde{\ell}$ be a component of $f^{-1}(\ell)$.
Applying the Hurwitz formula to the restriction morphism $f|_{\tilde{\ell}}$, we obtain that
\[2g -2 = -2 \deg(f|_{\tilde{\ell}}) + \deg(R),\]
where $g$ is the genus of $\tilde{\ell}$ and $R$ is the ramification divisor.
If $\deg(f|_{\tilde{\ell}}) = 2$, then $H\cdot \tilde{\ell} = \mathcal{O}_{\mathbb{P}^{n}}(1)\cdot f_{*}\tilde{\ell} = 2$.
Thus $\tilde{\ell}$ is a line on $X$ if and only if $g = 0$, $\deg(f|_{\tilde{\ell}}) = 1$, and $\deg(R) = 0$.
Thus the image $N$ of $f_{*}$ is the set of lines bitangent to $B$, or contained in $B$.
Note that there is a commutative diagram
\[
  \xymatrix{
    M' \ar[r]^{\mathrm{ev}} \ar[d]_{f_{*}} & X \ar[d]^{f} \\
    N' \ar[r]^{\mathrm{ev}_{N}} & \mathbb{P}^{n},
  }
\]
where $M' = \overline{M}_{0, 1}(X, 1)$ and $N'$ be the image of $f_{*} \colon M' \rightarrow \overline{M}_{0, 1}(\mathbb{P}^{n}, 1)$.
\begin{proof}[Proof of Theorem \ref{line}.(1): $H^{n} = 2$ case]
By \cite[II.3.11 Theorem]{Kollar1996}, there is an open subset $U \subset X$ such that any line on $X$ intersecting $U$ is free.
Thus, the fiber $\mathrm{ev}^{-1}(p)$ is smooth for any point $p \in U$.
\par
Since a smooth and connected scheme is irreducible, it suffices to show that the general fiber is connected.
Let $p \in X$ be a point such that $f(p) \notin B$. 
Then we have $\mathrm{ev}^{-1}(p) \cong \mathrm{ev}_{N}^{-1}(f(p))$.
We may assume that $f(p) = (1: 0: \dots : 0) \in \mathbb{P}^{n}$.
We will calculate the fiber $\mathrm{ev}_{N}^{-1}(f(p))$.
The polynomial $F$ defining the quartic $B$ has the form 
\[F(x_{0}, \dots , x_{n}) = x_{0}^{4} + x_{0}^{3}F_{1} + x_{0}^{2}F_{2} + x_{0}F_{3} + F_{4},\]
where $F_{i} \in k[x_{1}, \dots ,x_{n}]$ are homogeneous of degree $i$. 
Furthermore, by taking the linear transform $x_{0} \mapsto x_{0} - \frac{1}{4}F_{1}$, we may assume that $F_{1} = 0$.
We take a line $\ell = \{(s:a_{1}t:\dots:a_{n}t) \mid (s:t) \in \mathbb{P}^{1}\}$ for $a = (a_{1}:\dots:a_{n}) \in \mathbb{P}^{n-1}$.
Then $\ell$ is an element of $N$ if and only if every solution of the equation 
\[s^{4} + F_{2}(a)s^{2}t^{2} + F_{3}(a)st^{3} + F_{4}(a)t^{4} = 0\]
has an even multiplicity, that is, there exist $\alpha , \beta, \gamma \in k$ such that 
\[s^{4} + F_{2}(a)s^{2}t^{2} + F_{3}(a)st^{3} + F_{4}(a)t^{4} = (\alpha s^{2} + \beta st + \gamma t^{2})^{2}.\]
Comparing the coefficients, we obtain the equations
\begin{align*}
1 &= \alpha^{2}, \\
0 &= 2\alpha\beta, \\
F_{2} &= \beta^{2} + 2\alpha\gamma, \\
F_{3} &= 2\beta\gamma, \\
F_{4} &= \gamma^{2}.
\end{align*}
Thus, we have 
\[F_{3}(a) = F_{2}^{2}(a) - 4F_{4}(a) = 0.\]
Thus we obtain that $\mathrm{ev}^{-1}(p) \cong \mathrm{ev}_{N}^{-1}(f(p)) \cong V(F_{3}, F_{2}^{2} - 4F_{4}) \subset \mathbb{P}^{n - 1}$, which is connected.
\end{proof}
\begin{proof}[Proof of Theorem \ref{line}.(2): $H^{n} = 2$ case]
First, we give a characterization of a point $p \in f^{-1}(B)$ such that the fiber dimension is $n-2$.
In fact, we will show that for a point $p\in f^{-1}(B)$, $\dim \mathrm{ev}^{-1} = n-2$ if and only if $B\cap T_{f(p)}B$ is a cone with the vertex $f(p)$.
We may assume that $f(p) = (1:0:\dots:0) \in B$ and the tangent hyperplane to $B$ at $f(p)$ is given by $(x_{n} = 0)$.
Then $B$ is defined by a homogeneous polynomial
\[F(x_0, \dots , x_n) = x_{0}^{3}x_{n} + x_{0}^{2}F_{2} + x_{0}F_{3} + F_{4}. \]
where $F_{i} \in k[x_1,\dots,x_n]$ is a homogeneous polynomial of degree $i$.
We take a line $\ell = \{ (s:a_{1}t:\dots:a_{n}t) \mid (s:t) \in \mathbb{P}^{1}\}$ for $a = (a_{1}:\dots:a_{n}) \in \mathbb{P}^{n-1}$.
Then $\ell$ is an element of $N$ if and only if every solution of the equation
\[a_{n}s^{3}t + F_{2}(a)s^{2}t^2 + F_{3}(a)st^3 + F_{4}(a)t^4 = 0\]
has an even multiplicity, and this is equivalent to 
\[a_{n} = F_{3}(a)^2 - 4F_{2}(a)F_{4}(a) = 0.\]
Therefore the space of lines passing through $f(p)$ is isomorphic to the closed subset $V(x_{n}, F_{3}^{2} - 4F_{2}F_{4}) \subset \mathbb{P}^{n-1}$ of dimension at most $n - 2$.
Set $R_{i} = F_{i}(x_{1},\dots,x_{n-1},0)$ and $R = F(x_{0},\dots,x_{n-1},0)$. 
Hence $R_{i}$ is a homogeneous polynomial of degree $i$ with variables $x_{1},\dots,x_{n-1}$ and 
\[R = x_{0}^{2}R_{2} + x_{0}R_{3} + R_{4}.\]
We consider the relation \[R_{3}^{2} = 4R_{2}R_{4}\] as polynomials of $k[x_{1}. \dots , x_{n-1}]$.
We first assume that $R_{3} \neq 0$.
If $R_2$ divides $R_3$, then there is a homogeneous polynomial $P$ of degree $1$ such that $R_{3} = 2PR_{2}$ and $4P^{2}R_{2}^{2} = 4R_{2} R_{4}$, hence we have $R_{4} = P^{2}R_{2}$.
Substituting them, we obtain that
\begin{align*}
R &= x_{0}^{2}R_{2} + 2x_{0}PR_{2} + P^{2}R_{2}\\
  &= (x_{0} + P)^{2}R_{2}.
\end{align*}
On the other hand, if $R_2$ does not divide $R_3$, then there is a homogeneous polynomials $P$, $Q$ with $\deg P = 1$, $\deg Q = 2$ such that $R_2 = P^2$, $R_3 = 2PQ$, and $R_4 = Q^2$.
Indeed, $R_{2}$ is not an irreducible polynomial since $R_{2}$ does not divide $R_{3}$.
Write $R_{2} = P\hat{P}$, where $P$ and $\hat{P}$ are homogeneous polynomials of degree $1$.
We have $R_{3}^{2} = 4P\hat{P}R_{4}$.
If $P \ne \hat{P}$, then $R_{2}$ divides $R_{4}$, hence also $R_{3}$, a contradiction.
Thus $R_{2} = P^{2}$ and one can also see the existence of $Q$.
Thus, we obtain that 
\begin{align*}
R &= x_{0}^{2}P^{2} + 2x_{0}PQ + Q^{2}\\
  &= (x_{0}P + Q)^2.
\end{align*}
Therefore, in both cases, $R$ is not an irreducible polynomial, which is a contradiction by Lemma \ref{sectionirr}, hence $R_{3} = 0$.
Hence we obtain that $R_{2} = 0$ or $R_{4} = 0$.
If $R_{4} = 0$, then we have
\[R = x_{0}^{2}R_{2},\]
which is a contradiction by Lemma \ref{sectionirr} again.
Therefore $R_{2} = R_{3} = 0$ and $R = R_{4}$.
This implies that the tangent hyperplane section of $B$ at $f(p)$ is a cone with vertex at $f(p)$.
Thus there are only finitely many such points by a similar argument of the proof of  \cite[Corollary 2.2]{Coskun2009}.
\par
It remains to show that there are only finitely many points outside $f^{-1}(B)$ where the fiber dimension may jump.
Assume that there is a proper irreducible curve $C$ on $X$ where the fiber dimension is $n-2$.
Then $f(C) \not\subset B$ by the above argument. 
Let $Z$ be the reduced subscheme of $X$ covered by lines intersecting with $C$. 
The subscheme $Z$ is the union of all subschemes $Z_{p}$ which are spanned by lines through $p\in C$.
Since each $Z_{p}$ has dimension $n-1$ by Lemma \ref{spanned}, hence $\dim Z \ge n-1$.
If $\dim Z = n$, then the family of lines intersecting $C$ covers $X$.
However, those lines are not free, which is a contradiction.
Hence we have $\dim Z = n-1$.
Now, for a point $p \in C$, we let $D_{p}$ be the reduced divisor on $X$ which is covered by lines through $p$.
Then we claim that: 
\begin{claim}\label{maxdiv}
There is a nonempty divisor $D$ on $X$ such that $D \subset D_{p}$ for any point $p \in C$.
We set $D^{C}$ to be the maximal divisor satisfying this property.
\end{claim}
\begin{proof}
First, it follows that there is a divisor $D$ on $X$ such that $D \subset D_{p}$ for any point $p$ in an open subset $U \subset C$ since $D\subset Z$ and $\dim Z = n-1$.
We will show that $U=C$.
We define a subset $\Gamma \subset C \times D$ by 
\[\Gamma = \{(p,q)\in C \times D \mid \mbox{there exists a line $\ell$ through $p, q$}\}.\]
Consider the following commutative diagram
\[
  \xymatrix{
  	 & \overline{M}_{0,2}(X, 1)\ar[dl]_-{s_{1}}\ar[d]^-{s}\ar[dr]^-{s_{2}} & \\
    X & X \times X\ar[l]^-{p_{1}}\ar[r]_-{p_{2}} & X,
  }
\]
where $s$ is the evaluation map, $p_{i}$ is the $i$-th projection.
Then we have $\Gamma = s(s_{1}^{-1}(C) \cap s_{2}^{-1}(D))$.
Hence it is closed in $C \times D$.
On the other hand, $\Gamma$ contains $U \times D$.
Thus $\Gamma = C \times D$, i.e., $D \subset D_{p}$ for any $p\in C$. 
\end{proof}
Since $C \not\subset f^{-1}(B)$, we take two points $p, q \in C$ such that $f(p) \in B$ and $f(q) \notin B$.
Any line on $\mathbb{P}^{n}$ tangent to $B$ at $f(p)$ is contained in $T_{f(p)}B \cong \mathbb{P}^{n-1}$.
Since $\dim \mathrm{ev}^{-1}(p) = n-2$, we have $f(D_{p}) = T_{f(p)}B$, which is irreducible. 
Since $f(D^{C}) \subset f(D_{p})$ are $(n-1)$-dimensional, we have $f(D^{C}) = T_{f(p)}B$. 
Now consider the projection morphism $\pi \colon T_{f(p)}B\setminus \{f(q)\} \rightarrow \mathbb{P}^{n-2}$ from $f(q)$.
The restriction morphism to $B \cap T_{f(p)}B$ is finite since $f(q) \notin B$.
Let $\ell$ be a line such that $f(q)\in \ell \subset T_{f(p)}B$.
Then $\ell$ is bitangent to $B$ if and only if the fiber of $\pi|_{B \cap T_{f(p)}B}$ over the point $\pi(\ell)$ is non-reduced.
Hence, in order to obtain $f(D_{q}) = T_{f(p)}B$, any fiber of this restriction has to be non-reduced.
Thus $B \cap T_{f(p)}B$ also has to be non-reduced.
However, by Lemma \ref{sectionirr}, $B \cap T_{f(p)}B$ is reduced, a contradiction.
Hence we have $f(D_{q}) \neq f(D_{p}) = f(D^{C})$ and thus there are only finitely many points where the fiber does not have the expected dimension.
\end{proof}
\subsection{Proof of Theorem \ref{line}: $H^{n} = 1$ case}
Let $X$ be a del Pezzo manifold of degree $1$.
Then $X$ is a smooth hypersurface of degree $6$ in the weighted projective space $\mathbb{P}(1^n,2,3)$. 
Using a homogeneous polynomial $F = y^3 + F_{4}y + F_{6}$ of degree 6, where $F_{i} \in k[x_{1},\dots , x_{n-1}]$ is a homogeneous polynomial of degree $i$, and $\deg y = 2$, 
$X$ is written as 
\[X = V(z^{2} + F), \]
where $\deg z = 3$.
Then there is a double cover $f\colon X \rightarrow \mathbb{P} := \mathbb{P}(1^n,2)$ branched over a smooth sextic $Y = V(F)$ and the vertex $v := (0:\dots :0:1)\in \mathbb{P}$.
Let $Z := V(z)$ be the ramification divisor of $f$.
Set $w := (0:\dots:0:1:1)\in X$, which is the unique point mapping to the vertex $v \in \mathbb{P}$.
For later use, we introduce some notations.
Consider the projection $g \colon \mathbb{P}\setminus \{v\} \rightarrow \mathbb{P}^{n-1}$ from the vertex $v$.
The restriction $g|_{Y}$ is finite of degree $3$.
Let $R \subset Y$ be the ramification locus of $g|_{Y}$ and $B \subset \mathbb{P}^{n-1}$ be the branch locus.
More explicitly, these are written as 
\begin{align*}
R &= V(3y^{2} + F_{4}, 2y^{3} - F_{6})\subset Y, \\
B &= V(4F_{4}^{3} + 27F_{6}^{2})\subset \mathbb{P}^{n-1}. 
\end{align*}
In addition, we set $T := V(F_{4},F_{6})$, over which $g|_{Y}$ is totally ramified.
Therefore, there is a following diagram:
\[
  \xymatrix{
    X\ar[d]_-{f}\ar@{}[r]|*{\supset} & Z\ar[d]^-{\cong} & &\\
    \mathbb{P}\ar@{.>}[dr]_-{g}\ar@{}[r]|*{\supset} &Y\ar[d]^-{g|_{Y}}\ar@{}[r]|*{\supset} & R\ar[d] &\\
     & \mathbb{P}^{n-1}\ar@{}[r]|*{\supset} & B\ar@{}[r]|*{\supset} & T.
  }
\]
Let $\tilde{\ell}$ be a line on $X$. 
Then $\mathcal{O}_{\mathbb{P}}(1) \cdot f_{*}\tilde{\ell} = \mathcal{O}_{X}(1)\cdot \tilde{\ell} = 1$ by the projection formula. 
Consider the restriction morphism $f|_{\tilde{\ell}}$.
We will call a connected $1$-cycle $\ell \subset \mathbb{P}$ to be a line if $\mathcal{O}_{\mathbb{P}}(1)\cdot \ell = 1$.
\begin{enumerate}[(1)]
\item If $\deg f|_{\tilde{\ell}} = 1$, then $\alpha := f_{*}\tilde{\ell}$ is a smooth line not passing through the vertex $v$.
By the Hurwitz formula, it is shown that $\alpha$ is tritangent to $Y$ or $\alpha \subset Y$.
Moreover, if $\alpha \not\subset Y$, then the pull back $f^{*}(\alpha)$ is a union of two distinct lines on $X$.
Since $v\notin \alpha$, we have $\mathcal{O}_{\mathbb{P}^{n-1}}(1)\cdot g_{*}\alpha = \mathcal{O}_{\mathbb{P}}(1) \cdot \alpha = 1$ by the projection formula.
Hence $\deg g|_{\alpha} = 1$ and $g_{*}\alpha$ is a line on $\mathbb{P}^{n-1}$.
\item If $\deg f|_{\tilde{\ell}} = 2$, then $\alpha = 2\beta$ is a double line, where $\beta$ is the curve class satisfying $\mathcal{O}_{\mathbb{P}}(1) \cdot \beta = \frac{1}{2}$ and $v \in \beta$. 
Moreover, $\beta$ is tangent to $Y$.
\end{enumerate}
Let $N_{1}$ (resp. $N_{2}$) be the locus of smooth lines of type (1) (resp. double lines of type (2)), and set $N:= N_{1} \cup N_{2}$.
Since the inverse image of a line in $N_{2}$ is a line on $X$ passing through $w$, the dimension of $N_{2}$ is at most $n-2$ by Lemma \ref{unexpected}.
Hence $N_{2}$ cannot form a component of $N$ since $\dim N = \dim M_{0,0}(X,1) \ge 2n-4$.
Moreover, a general element of $N$ has type (1) and not contained in $Y$.
Hence the pushforward of $f$ induces the double cover $f_{*} \colon M:= \overline{M}_{0, 0}(X, 1) \rightarrow N$ and the commutative diagram
\[
  \xymatrix{
    M'\ar[r]^-{\mathrm{ev}}\ar[d]_-{f_{*}} & X\ar[d]^-{f} \\
    N'\ar[r]^-{\mathrm{ev}_{N}} & \mathbb{P},
  }
\]
where $M' = \overline{M}_{0, 1}(X, 1)$ and $N'$ be the image of $f_{*} \colon M' \rightarrow \overline{M}_{0, 1}(\mathbb{P}, 1)$.
We introduce two lemmas for the proof of Theorem \ref{line}.(2):
\begin{lem}\label{isolated1}
Notation as above.
Then any fundamental divisor $D \in |H|$ on $X$ has only isolated singularities.
\end{lem}
\begin{proof}
One can prove this in a similar way to \cite[Lemma 6.6]{Lehmann2018}.
We give a proof for completeness.
\par
Since $\mathcal{O}_{X}(H) \cong \mathcal{O}_{\mathbb{P}(1^{n},2,3)}(1)|_{X}$, we may assume that $D$ is cut out by a hyperplane $V(x_{0})$.
For each integer $0\le i\le n-1$, we set $U_{i} := (x_{i} \neq 0) \subset \mathbb{P}(1^{n},2,3)$, which is an open affine subset isomorphic to 
\[\mathbb{A}^{n+1} \cong \Spec \left(k\left[\frac{x_{0}}{x_{i}}, \dots, \frac{x_{n-1}}{x_{i}}, \frac{y}{x_{i}^{2}}, \frac{z}{x_{i}^{3}}\right]\right).\]
Set $\tilde{x}_{j} := x_{j}/x_{i}$ for $j\neq i$, $\tilde{y} := y/x_{i}^{2}$, and $\tilde{z} := z/x_{i}^{3}$.
Then we have $D \subset \bigcup_{i=1}^{n-1}U_{i} \cup \{w\}$.
For simplicity, we assume we are in $U_{n-1}$.
Then the Jacobian matrix for $D \cap U_{n-1}$ is 
\[
  \begin{pmatrix}
    1 & 0 & \cdots & 0 & 0 & 0\\
    \frac{\partial \tilde{F}}{\partial \tilde{x}_{0}} & \frac{\partial \tilde{F}}{\partial \tilde{x}_{1}} & \cdots & \frac{\partial \tilde{F}}{\partial \tilde{x}_{n-2}} & \frac{\partial \tilde{F}}{\partial \tilde{y}} & 2\tilde{z}
  \end{pmatrix}, 
\]
where $\tilde{F} = F/x_{n-1}^{6} \in k[\tilde{x}_{0}, \dots, \tilde{x}_{n-2}, \tilde{y}, \tilde{z}]$.
Then the singular locus $\mathrm{Sing}(D\cap U_{n-1})$ is contained in $V(\tilde{z})$.
On the other hand, $D\cap U_{n-1}$ is smooth along $V(\frac{\partial \tilde{F}}{\partial \tilde{x_{0}}}, \tilde{z})$ since $Z$ is smooth.
Running the same argument for any $i \in \{1,\dots, n-2\}$, we see that $\mathrm{Sing}(D) \subset Z \cup \{w\}$ and $D$ is smooth along $V(\frac{\partial F}{\partial x_{0}})\cap Z$. 
However, $\mathrm{Sing}(D)$ meets with $V(\frac{\partial F}{\partial x_{0}})\cap Z$ unless the locus is $0$-dimensional.
Thus the claim holds.
\end{proof}
\begin{lem}\label{isolated3}
Notation as above.
Let $D\in |3H|$ be a member of the form $D = V(z + P) \cap X$, where $P \in k[x_{0},\dots,x_{n-1},y]$ is a homogeneous polynomial of degree $3$.
Then $D$ has only isolated singularities.
\end{lem}
\begin{proof}
As in the proof of Lemma \ref{isolated1}, we use the notation $U_{i}, \tilde{x}_{i}, \tilde{y}, \tilde{z}$, and $\tilde{F}$.
Since the degree of $P$ is odd, we see that $w \notin D$. 
Hence we have $D \subset \bigcup_{i=0}^{n-1}U_{i}$.
For simplicity, we assume we are in $U_{n-1}$.
Then the Jacobian matrix for $D\cap U_{n-1}$ is 
\[
  \begin{pmatrix}
    \frac{\partial \tilde{P}}{\partial \tilde{x}_{0}} & \cdots & \frac{\partial \tilde{P}}{\partial \tilde{x}_{n-2}} & \frac{\partial \tilde{P}}{\partial \tilde{y}} & 1\\
    \frac{\partial \tilde{F}}{\partial \tilde{x}_{0}} & \cdots & \frac{\partial \tilde{F}}{\partial \tilde{x}_{n-2}} & \frac{\partial \tilde{F}}{\partial \tilde{y}} & 2\tilde{z}
  \end{pmatrix}, 
\]
where $\tilde{P} = P/x_{n-1}^{3} \in k[\tilde{x}_{0}, \dots, \tilde{x}_{n-2}, \tilde{y}, \tilde{z}]$.
Then $D \cap U_{n-1}$ is smooth along $V(\tilde{z})$ since $Z$ is smooth.
Running the same argument for any $i \in \{0,\dots, n-2\}$, we see that $D$ is smooth along $Z$.
However, the singular locus of $D$ meets with $Z$ unless the locus is $0$-dimensional.
Thus the claim holds.
\end{proof}
\begin{proof}[Proof of Theorem \ref{line}.(1): $H^{n} = 1$ case]
As in the case of $H^{n}=2$, the general fiber of $\mathrm{ev}$ is smooth by \cite[II.3.11 Theorem]{Kollar1996}.
Hence it suffices to show the connectedness of the general fiber.
Let $p \in X$ map to a point $f(p) \notin Y \cup g^{-1}(B)$. 
Then we have $\mathrm{ev}^{-1}(p) \cong \mathrm{ev}_{N}^{-1}(f(p))$ and any member of $N_{2}$ cannot pass through $f(p)$ since $f(p) \notin g^{-1}(B)$. 
We may assume that $f(p) = (1: 0:\dots :0:b_{0}) \in \mathbb{P}$.
Let $\ell$ be a line on $\mathbb{P}$ through $f(p)$ but not through $v$.
We recall that a line on $\mathbb{P}$ is a connected $1$-cycle $\alpha$ such that $\mathcal{O}_{\mathbb{P}}(1)\cdot \alpha = 1$.
Hence by the projection formula, we see that $\deg g|_{\ell} = 1$ and $g(\ell)$ is a line on $\mathbb{P}^{n-1}$. 
Since $gf(p) = (1:0:\dots :0) \in g(\ell)$, one can write 
\[g(\ell) = \{(s:a_{1}t:\dots:a_{n-1}t) \mid (s:t)\in \mathbb{P}^{1}\}\]
for $(a_{1}:\dots :a_{n-1})\in \mathbb{P}^{n-2}$.
Thus $\ell$ has the form 
\[\ell = \{(s:a_{1}t:\dots :a_{n-1}t:b_{0}s^2 + b_{1}st + b_{2}t^2) \mid (s:t)\in \mathbb{P}^{1}\}\]
for $(a_{1}:\dots :a_{n-1}:b_{1}:b_{2}) \in \mathbb{P}(1^n, 2)$.
We write the polynomial $F(\ell)$ in the variables $s,t$, substituting the parameter of $\ell$ in $F$, as 
\[F(\ell) = \sum_{i=0}^{6} G_{i}s^{6-i}t^{i}.\]
Since $\ell$ contains $f(p)$, $\ell$ cannot be an element of $N_{2}$.
Then $\ell$ is an element of $N$ if and only if every solution of the equation 
\[F(\ell) = 0\]
has an even multiplicity, that is, there exist $\alpha, \beta, \gamma, \delta \in k$ such that 
\[F(\ell) = (\alpha s^{3} + \beta s^{2}t + \gamma st^{2} + \delta t^{3})^{2}.\]
Comparing the coefficients, we obtain the equations
\begin{align*}
G_{0} &= \alpha^{2}, \\
G_{1} &= 2\alpha\beta, \\
G_{2} &= \beta^{2} + 2\alpha\gamma, \\
G_{3} &= 2\alpha\delta + 2\beta\gamma, \\
G_{4} &= \gamma^{2} + 2\beta\delta, \\
G_{5} &= 2\gamma\delta, \\
G_{6} &= \delta^{2}.
\end{align*}
By assumption, $G_{0} = F(f(p)) \ne 0$, hence we may assume that $\alpha = 1$.
Then we obtain 
\begin{align*}
\beta &= \frac{1}{2}G_{1}\\
\gamma &= \frac{1}{2}G_{2} - \frac{1}{8}G_{1}^{2}\\
\delta &= \frac{1}{2}G_{3} - \frac{1}{4}G_{1}G_{2} + \frac{1}{16}G_{1}^{3}.
\end{align*}
Substituting them into $G_{4}$, $G_{5}, G_{6}$, we obtain three equations with variables $a_{0}, \dots, a_{n-1}, b_{1}, b_{2}$.
Thus $\mathrm{ev}_{N}^{-1}(f(p))$ is defined by three equations in $\mathbb{P}(1^{n}, 2)$. 
Since $\dim \mathrm{ev}^{-1}(p) = n-3$ except for finitely many points by Theorem \ref{line}.(2) proved later, the general fiber is a complete intersection, hence connected. 
\end{proof}
\begin{proof}[Proof of Theorem \ref{line}.(2): $H^{n} = 1$ case]
For each point $p \in X$ such that $\dim \mathrm{ev}^{-1}(p) = n-2$, 
we let $D_{p}$ be the reduced divisor on $X$ which is covered by lines through $p$.
Then we claim that: 
\begin{claim}\label{divsing}
For any irreducible component $D \subset D_{p}$, $D$ is singular at $p$.
\end{claim} 
\begin{proof}
Suppose that $p$ is a smooth point of $D$.
Let $\phi \colon \tilde{D} \rightarrow D$ be a smooth resolution.
Let $s \colon \overline{M}_{0,1}(\tilde{D}, \alpha)\rightarrow \tilde{D}$ be the evaluation map, where $\alpha$ is the class of the strict transforms of lines.
Since $D$ is smooth at $p$, the fiber $\phi^{-1}(p)$ consists of a single point $\tilde{p}$, and $\dim s^{-1}(\tilde{p}) = \dim \mathrm{ev}^{-1}(p) = n-2$.
Moreover, since curves parametrized by $s^{-1}(\tilde{p})$ dominate $\tilde{D}$, there exists a free curve $[\bar{\ell}] \in s^{-1}(\tilde{p})$ by \cite[II.3.11 Theorem]{Kollar1996}.
Thus 
\[\dim_{[\bar{\ell}]} s^{-1}(\tilde{p}) = \dim_{[\bar{\ell}]} \overline{M}_{0,1}(\tilde{D}, \alpha) - \dim \tilde{D}\]
by \cite[II.3.5.4 Corollary]{Kollar1996}.
On the other hand, \cite[II.1.2 Theorem]{Kollar1996} yields that 
\[\dim_{[\bar{\ell}]} \overline{M}_{0,0}(\tilde{D}, \alpha) = -K_{\tilde{D}}\cdot \alpha + \dim \tilde{D} - 3. \]
Hence we obtain that $-K_{\tilde{D}}\cdot \alpha = n$.
Then $(K_{\tilde{D}}+n\phi^{*}H)\cdot \alpha = 0$.
Thus $K_{\tilde{D}}+n\phi^{*}H \in \partial \overline{\mathrm{Eff}}^{1}(X)$ and hence $a(D, H) = n > a(X, H)$. 
However, this is impossible by Lemma \ref{largea}.
\end{proof}
Suppose that there is a proper irreducible curve $C$ on $X$ such that for any $p\in C$, $\dim \mathrm{ev}^{-1}(p) = n-2$.
Let $D^{C}$ be the divisor as defined in Claim \ref{maxdiv}.
By Claim \ref{divsing}, we see that any component $D \subset D^{C}$ is singular along the curve $C$.
In the remaining part of the proof, we show that the divisor $D^{C}$ cannot exist dividing into two cases: (i) $f(C) \subset g^{-1}(B) \cup Y$, or (ii) $f(C) \not\subset g^{-1}(B) \cup Y$.
\begin{enumerate}[(i)]
\item 
Suppose that $f(C) \subset g^{-1}(B) \cup Y$.
If $f(C) \subset Y$, then $f(C)$ intersects with $R$.
Also when $f(C) \subset g^{-1}(B)$, one can show that $f(C)$ intersects with $R$.
Indeed, if $f(C)$ is contracted by $g$, then the assertion is clear since $f(C) \subset g^{-1}(B)$.
Suppose $gf(C) \subset B$ is a curve.
Then it suffices to show that $gf(C)$ intersects with $T$.
Since $V(F_{6})$ is an ample divisor on $\mathbb{P}^{n-1}$, one can take a point $a \in gf(C) \cap V(F_{6})$.
Then we have $F_{4}(a) = 0$ since $gf(C) \subset B$.
Hence $a \in gf(C) \cap T$.
\par
For a point $p \in C$ such that $f(p) \in R$, we will analyze the divisor $D_{p}$.
Consider the morphism $\gamma \colon R \rightarrow (\mathbb{P}^{n-1})^{*}$ defined by 
\[\gamma (r) = \left(\frac{\partial F}{\partial x_{0}}(r): \dots :\frac{\partial F}{\partial x_{n-1}}(r)\right),\]
which is well-defined since $Y$ is smooth and $\frac{\partial F}{\partial y}(r) = 0$ for any $r \in R$.
We view $\gamma(r)$ also as the corresponding hyperplane in $\mathbb{P}^{n-1}$.
Then for any line $\ell \subset \mathbb{P}$ tangent to $Y$ at $f(p)$, 
the image $g(\ell)$ is contained in the hyperplane $\gamma(f(p)) \cong \mathbb{P}^{n-2}$. 
Since the dimension of $D_{p}$ is $n-1$, it follows from Lemma \ref{spanned} that $f(D_{p})$ corresponds to $g^{-1}(\gamma(f(p)))$ in $\mathbb{P}$. 
Therefore, $D_{p}$ is a fundamental divisor on $X$.
Since $D_{p}$ is irreducible, we see that $D^{C} = D_{p}$.
Then $D_{p}$ is singular along the curve $C$ by Claim \ref{divsing}.
However, it contradicts Lemma \ref{isolated1}.
\item
Suppose that $f(C) \not\subset g^{-1}(B) \cup Y$.
Note that for any $p\in X$ such that $f(p) \notin g^{-1}(B)$, $D_{p}$ cannot contain the point $w$ since a line $\ell \in N$ contains $v$ if and only if $\ell \in N_{2}$.
Therefore, $D^{C}$ cannot contain any divisor which is the pullback of a divisor on $\mathbb{P}^{n-1}$.
Let $D \subset D^{C}$ be an irreducible component.
Then $gf(D) = \mathbb{P}^{n-1}$.
Let $\iota \colon X \rightarrow X$ be the involution defined by 
\[(x_{0}:\dots:x_{n-1}:y:z) \mapsto (x_{0}:\dots:x_{n-1}:y:-z).\]
Note that $f = f \circ \iota$ and $\iota|_{Z} = \mathrm{id}_{Z}$.
Since $w\notin D$, we see that $\deg f|_{\tilde{\ell}} = 1$ for any line $\tilde{\ell} \subset D$.
Hence if $q \in C\setminus Z$, then $\iota(q) \notin D$.
Therefore $D \neq \iota(D)$.
Now we fix a point $p \in C \cap Z$ and let $M_{D}$ be the irreducible component of $\mathrm{ev}^{-1}(p)$ which parametrizes lines dominating $D$.
Then we obtain that 
\[M_{D} \cap M_{\iota(D)} \subset \{\mbox{lines contained in $Z$}\}\]
since $D$ and $\iota(D)$ are different components of $D_{p}$.
Furthermore, we claim that:
\begin{claim}\label{inv}
$D \cap \iota(D) \subset Z$.
\end{claim}
\begin{proof}
Pick a point $q \in D\setminus Z$.
It suffices to show that $q \notin \iota(D)$.
Since $p \in C$ and $q \in D$, one can take a line $\tilde{\ell} \subset D$ through both $p$ and $q$.
The image $\ell := f(\tilde{\ell})$ is a member of $N$ such that $\{f(p), f(q)\} \subset \ell \subset f(D)$.
Since $v \notin f(D)$, $\ell$ is a member of $N_{1}$.
We now assume that there exists a line $\ell' \in N_{1}\setminus \{\ell\}$ such that $f(p) \in \ell' \subset f(D)$.
If $g(\ell) \neq g(\ell')$, then $\ell'$ cannot contain $f(q)$ since $g(\ell) \cap g(\ell') = \{gf(p)\} \neq \{gf(q)\}$.
Suppose that $g(\ell) = g(\ell')$.
If we view $\ell$ and $\ell'$ as the curves in $\mathbb{P}(1^{2},2) \cong g^{-1}(g(\ell))$, then $\ell, \ell' \in |\mathcal{O}_{\mathbb{P}(1^{2},2)}(2)|$ and hence we have the intersection number $\ell \cdot \ell' = \frac{1}{2}\cdot 2^{2} = 2$ on the surface $\mathbb{P}(1^{2},2)$.
On the other hand, both $\ell$ and $\ell'$ are tangent to $Y$ at $f(p)$ since $\ell, \ell' \in N_{1}$.
Hence $\ell$ and $\ell'$ intersect at $f(p)$ with multiplicity $2$.
Thus $\ell'$ cannot contain $f(q)$.
This implies that $\ell$ is the unique line such that $\{f(p), f(q)\} \subset \ell \subset f(D)$.
The inverse image of $\ell$ is the union of two lines $\tilde{\ell}, \iota(\tilde{\ell})$.
Since $q \in \tilde{\ell} \setminus Z$, we see that $q \notin \iota(\tilde{\ell})$.
We also see that $\iota(\tilde{\ell}) \notin M_{D}$ since 
\[\iota(\tilde{\ell}) \in M_{\iota(D)} \setminus \{\mbox{lines contained in $Z$}\}. \]
Therefore, $\iota(D)$ cannot contain $q$, hence the claim holds.
\end{proof}
\par
We retake a point $p \in C\setminus Z$.
Let $\tilde{\ell}$ be a line on $X$ such that $p \in \tilde{\ell} \subset D$.
Then we claim that:
\begin{claim}\label{3H}
$D \cdot \iota(\tilde{\ell}) = 3$.
\end{claim}
\begin{proof}
By Claim \ref{inv}, we have 
\[\tilde{\ell} \cap \iota(\tilde{\ell}) \subset D \cap \iota(\tilde{\ell}) \subset D \cap \iota(D) \subset Z.\]
Set $\ell := f(\tilde{\ell}) = f(\iota(\tilde{\ell}))$.
Since $f(D \cap \iota(\tilde{\ell})) \subset \ell$, we have 
\[D \cap \iota(\tilde{\ell}) \subset f^{-1}(\ell) \cap Z = \tilde{\ell} \cap \iota(\tilde{\ell}).\]
Thus we obtain a set-theoretic equality $\tilde{\ell} \cap \iota(\tilde{\ell}) = D \cap \iota(\tilde{\ell})$.
Via an isomorphism $\mathbb{P}^{1} \cong \ell$, we obtain $\mathbb{P}(1^{2},3) \cong \bar{f}^{-1}(\ell)$, where $\bar{f} \colon \mathbb{P}(1^{n},2,3)\dashrightarrow \mathbb{P}(1^{n},2)$ is the projection map.
If we view $\tilde{\ell}$ and $\iota(\tilde{\ell})$ as the curves in $\mathbb{P}(1^{2}, 3)$, then $\tilde{\ell}, \iota(\tilde{\ell}) \in |\mathcal{O}_{\mathbb{P}(1^{2}, 3)}(3)|$ and hence we have the intersection number $\tilde{\ell} \cdot \iota(\tilde{\ell}) = \frac{1}{3}\cdot 3^{2} = 3$ on the surface $\mathbb{P}(1^{2}, 3)$.
Thus, in order to prove that $D\cdot \iota(\tilde{\ell}) = 3$, it is enough to show that the intersection multiplicity at each point $q_{0} \in D \cap \iota(\tilde{\ell})$ is equal to that of $q_{0} \in \tilde{\ell} \cap \iota(\tilde{\ell})$.
\par
Suppose that there is a line $\ell' \in N_{1}\setminus \{\ell\}$ such that $f(q_{0}) \in \ell' \subset f(D)$.
We assume that $g(\ell) = g(\ell')$. 
As in the proof of Claim \ref{inv}, we see that $f(p) \notin \ell'$. 
We recall that $f(D)$ is covered by lines in $N_{1}$ passing through $f(p)$.
We also recall that any line $\alpha \in N_{1}$ satisfies $\deg g|_{\alpha} = 1$.
Hence we see that $g^{-1}(gf(p)) \cap f(D) = \{f(p)\}$.
This implies that the unique point in $g^{-1}(gf(p)) \cap \ell'$ is not contained in $f(D)$, which contradicts the assumption $\ell' \subset f(D)$.
Therefore, $g(\ell) \neq g(\ell')$.
Then $\ell$ and $\ell'$ intersect at $f(q_{0})$ with different tangent directions.
Summarizing the argument, for any line $\tilde{\ell}'$ on $X$ distinct from $\tilde{\ell}$ such that $q_{0} \in \tilde{\ell}' \subset D$, the intersection multiplicity at $q_{0} \in \tilde{\ell}' \cap \iota(\tilde{\ell})$ is $1$.
Therefore, the intersection multiplicity at $q_{0} \in D \cap \iota(\tilde{\ell})$ is equal to that of $q_{0} \in \tilde{\ell} \cap \iota(\tilde{\ell})$, as required.
\end{proof}
By Claim \ref{3H}, we have $D \in |3H|$.
Write 
\[D = V(cz + P) \cap X, \]
where $c \in k$ and $P \in k[x_{0}, \dots, x_{n-1}, y]$ is a homogeneous polynomial of degree $3$.
Since $D \neq \iota(D)$, we may assume that $c = 1$.
Therefore, by Lemma \ref{isolated3}, $D$ has only isolated singularities, which contradicts Claim \ref{divsing}.
\end{enumerate}
\end{proof}
\begin{rem}
Applying the above argument, one can give a very short proof for Theorem \ref{line}.(2) when $H^{n} = 2$:
if $C$ is a proper irreducible curve where the fiber has dimension $n-2$, then for any point $p \in C \cap f^{-1}(B)$, the divisor $D_{p}$ is a fundamental divisor.
Hence the rest is to prove that any fundamental divisor has only isolated singularities. 
\par
The reason why we gave the first proof is because one can see that for a point $p \in f^{-1}(B)$, $\dim \mathrm{ev}^{-1}(p) = n-2$ if and only if $f(p)$ is a cone point of $B$.
This result is similar to the characterization for smooth cubics $X = X_{3} \subset \mathbb{P}^{n+1}$:
for a point $p \in X$, $\dim \mathrm{ev}^{-1}(p) = n-2$ if and only if $p$ is an Eckardt point (\cite[Lemma 2.1 and Definition 2.3]{Coskun2009}).
Unfortunately, the author was not able to give a similar characterization for $H^{n} = 1$.
\end{rem}
\section{Rational curves of arbitrary degree}
In this section, we prove the movable bend and break and Theorem \ref{main} for arbitrary $d \ge 1$.
First, we can apply the argument in \cite[Proposition 2.5]{Coskun2009}, hence we obtain:
\begin{prop}\label{induction}
Let $X$ be a del Pezzo manifold of Picard rank $1$ and dimension $n \ge 4$ with an ample generator $H$. 
Let $\mathrm{ev}_d \colon \overline{M}_{0, 1}(X, d) \rightarrow X$ be the evaluation map for each integer $d \ge 1$.
Let $S \subset X$ be the finite set as in Theorem \ref{line}.
Then for any $d \ge 1$,
\begin{itemize}
\item if $p \notin S$, then $\dim (\mathrm{ev}^{-1}_{d}(p)) = (n - 1)d - 2$, and
\item if $p \in S$, then $\dim (\mathrm{ev}^{-1}_{d}(p)) \le (n - 1)d - 1$.
\end{itemize}
Furthermore, any component $M$ of $\overline{M}_{0,0}(X, d)$ generically parametrizes free curves and has the expected dimension $(n - 1)d + n - 3$.
\end{prop}
We will prove the movable bend and break introduced in \cite{Lehmann2019}.
We give a proof for completeness.
\begin{thm}[Movable bend and break]\label{MBB}
Let $X$ be a del Pezzo manifold of Picard rank $1$ and dimension $n \ge 4$ with an ample generator $H$. 
Then, any free curve deforms to a chain of free lines.
\end{thm}
\begin{proof}
Fix an integer $d \ge 1$.
Let $M \subset \overline{M}_{0,0}(X, d)$ be a component and we take a free curve $(C, f) \in M$.
By \cite[II.3.11 Theorem]{Kollar1996}, there is a closed subset $V \subsetneq X$ such that any non-free curve of degree at most $d$ is contained in $V$.
By Mori's bend and break lemma (e.g., \cite[II.5.5 Corollary]{Kollar1996}), the locus of stable maps with reducible domains has codimension $1$. 
Thus $f$ can be deformed to a stable map with a reducible domain $g \colon C_{1} + C_{2} \rightarrow X$.
Let $\Lambda \ni g$ be a component of $\overline{M}_{0, 1}(X, d_{1}) \times_{X} \overline{M}_{0, 1}(X,d_{2})$, where $d_{1} + d_{2} = d$.
We now assume that $\Lambda$ does not parametrize chains of two free curves.
We calculate the dimension of $\Lambda$ dividing into two cases.
\begin{enumerate}[(1)]
\item 
Suppose that any stable map in the image of $\Lambda \rightarrow \overline{M}_{0, 1}(X, d_{1})$ has a reducible domain.
Then by Proposition \ref{induction}, we have 
\begin{align*}
\dim \Lambda &\le (n - 1)d_{1} - 3 + (n - 1)d_{2} - 2 + n\\
 &= (n - 1)d + n - 5.
\end{align*}
\item
Suppose that general stable map in the image of $\Lambda \rightarrow \overline{M}_{0, 1}(X, d_{1})$ is irreducible but non-free.
Then the node $p \in C_{1} \cap C_{2}$ maps to a point in $V$.
Hence by Proposition \ref{induction}, we have 
\begin{align*}
\dim \Lambda &\le (n - 1)d_{1} - 2 + (n - 1)d_{2} - 2 + (n - 1)\\
 &= (n - 1)d + n - 5. 
\end{align*}
\end{enumerate}
Thus the codimension of $\Lambda$ in $M$ is greater than $1$, which means that general stable map with a reducible domain is a chain of two free curves.
Thus the assertion follows by \cite[Lemma 5.9]{Lehmann2019} and induction on the degree $d$. 
\end{proof}
\begin{proof}[Proof of Theorem \ref{main}]
For the case $d = 1$, we have proved it in Theorem \ref{line}.
Let $d > 1$ and let $M \subset \overline{M}_{0,0}(X, d)$ be a component.
Take a free curve $(C, f) \in M$.
Then by Theorem \ref{MBB}, $(C, f)$ deforms to a chain of free lines of length $d$.
Let $U \subset \overline{M}_{0, 0}(X, 1)$ be an open sublocus of free lines.
Then the fiber product 
\[\Delta := U' \times_{X} U'' \times_{X} \dots \times_{X} U'' \times_{X} U'\]
is the locus of chains of free lines, where $U' \subset \overline{M}_{0, 1}(X, 1)$ and $U'' \subset \overline{M}_{0, 2}(X, 1)$ are corresponding subloci.
By \cite[Lemma 5.7]{Lehmann2019}, the projections from $\Delta$ to each factor are dominant and flat.
In addition, by Theorem \ref{line}, the general fiber of the evaluation map $\mathrm{ev} \colon \overline{M}_{0, 1}(X, 1) \rightarrow X$ is irreducible.
Thus the locus $\Delta$ of chains of free lines is irreducible.
Then we see that $M$ is the unique component of $\overline{M}_{0,0}(X, d)$ containing $\Delta$, hence $\overline{M}_{0,0}(X,d) = M$.
\end{proof}
\section{Geometric Manin's Conjecture}
As a corollary of Theorem \ref{main}, we will prove Geometric Manin's Conjecture for our case.
In \cite{Lehmann2019}, the authors defined Manin components, which they propose should be counted in the conjectural asymptotic formula.
For details about this conjecture, see \cite{Tanimoto2021}.
We recall that a smooth projective variety $X$ is weak Fano if $-K_{X}$ is a nef and big divisor.
\begin{dfn}[\cite{Tanimoto2021} Definition 4.3]\label{ManinComp}
Let $X$ be a weak Fano manifold.
A generically finite morphism $f\colon Y \rightarrow X$ from a projective manifold is a breaking morphism if either
\begin{enumerate}[(i)]
\item $a(Y, -f^{*}K_{X}) > a(X, -K_{X})$, or
\item $f$ is an $a$-cover with Iitaka dimension $\kappa(Y,K_{Y}-f^{*}K_{X}) > 0$, or
\item $f$ is a face contracting morphism.
\end{enumerate}
An irreducible component $M \subset \Hom(\mathbb{P}^{1}, X)$ is an accumulating component if there is a breaking morphism $f \colon Y \rightarrow X$ and a component $N \subset \Hom(\mathbb{P}^{1}, Y)$ such that $f$ induces a dominant generically finite map $N \dashrightarrow M$.
A Manin component is a component which is not accumulating.
\end{dfn}
\begin{conj}[Geometric Manin's Conjecture (\cite{Tanimoto2021})]
Let $X$ be a weak Fano manifold.
Then there is an integer $m \ge 1$ and a nef integral $1$-cycle $\alpha \in \mathrm{Nef}_{1}(X)_{\mathbb{Z}}$ such that for any nef integral $1$-cycle $\beta \in \alpha + \mathrm{Nef}_{1}(X)_{\mathbb{Z}}$, the space $\Hom(\mathbb{P}^{1},X,\beta)$ has exactly $m$ Manin components.  
\end{conj}
\begin{lem}\label{noa-cover}
Let $X$ be a del Pezzo manifold of Picard rank $1$ and dimension $n \ge 4$. 
Then $X$ has no $a$-cover.
\end{lem}
\begin{proof}
Let $L$ be an ample sheaf on $X$ such that $\omega_{X} \cong L^{n-1}$.
Assume that there exists an $a$-cover $f\colon Y \rightarrow X$ with $\kappa := \kappa(Y,K_{Y} + a(Y, f^{*}L)f^{*}L)$.
By \cite[Theorem 3.15]{Lehmann2019a}, the general member $H \in |L|$ is a del Pezzo manifold of dimension $n-1$.
Since $f$ can be replaced by the composition with a birational morphism, we may assume that $H_{Y} := f^{-1}(H)$ is also smooth. 
In this situation, we obtain the isomorphism as in the proof of \cite[Theorem 3.15]{Lehmann2019a}:
\[H^{0}(Y,m(K_{Y} + a(Y,f^{*}L)f^{*}L))\rightarrow H^{0}(H_{Y},m(K_{H_{Y}} + (a(Y,f^{*}L) - 1)f^{*}L))\]
for sufficiently divisible integer $m \ge 1$.
Hence $a(H_{Y},f^{*}L) = a(Y,f^{*}L) - 1$ and $\kappa(H_{Y},K_{H_{Y}} + a(H_{Y},f^{*}L)f^{*}L) = \kappa(Y,K_{Y} + a(Y,f^{*}L)f^{*}L)$.
Thus the restriction $f|_{H_{Y}}$ is again an $a$-cover. 
Repeating this process, we eventually obtain an $a$-cover $f' \colon Y' \rightarrow X'$ to a del Pezzo threefold such that $\kappa(Y',K_{Y'} + a(Y',f^{*}L)f^{*}L) = \kappa$.
Then by the proof of \cite[Lemma 7.2]{Lehmann2019}, the Iitaka dimension $\kappa$ is equal to $2$.
\par
Let $\phi \colon Y \rightarrow W$ be the Iitaka fibration for $K_{Y} - f^{*}K_{X}$.
By the above argument, a general fiber $Y_{0}$ of $\phi$ has dimension $n-2$ and $a(Y_{0}, f^{*}L) = n-1$.
Then the image $f(Y_{0})$ also has $a(f(Y_{0}), H) = n-1$.
Let $\nu \colon \hat{Z} \rightarrow f(Y_{0})$ be the normalization.
Then we see that $(\hat{Z}, \nu^{*}H)$ is isomorphic to $(\mathbb{P}^{n-2},\mathcal{O}(1))$.
Then, in fact, $\nu$ is an isomorphism since any strict sublinear system of $|\mathcal{O}(1)|$ cannot define a morphism.
Hence $f(Y_{0})$ is isomorphic to $\mathbb{P}^{n-2}$.
Thus it suffices to prove that $X$ cannot be dominated by projective $(n-2)$-spaces.
Let $Z \subset X$ be a projective $(n-2)$-space.
For any point $p \in Z$, the space of lines in $Z$ passing through $p$ has dimension $n-3$.
Hence it is a component of the fiber of the evaluation map $\mathrm{ev}_{1} \colon \overline{M}_{0,1}(X,1) \rightarrow X$ over $p$ if $p \notin S$, where $S$ is the finite set as in Theorem \ref{line}.
However, the general fiber of $\mathrm{ev}_{1}$ is irreducible, and the lines in such a fiber do not cover $Z$ by the explicit description of the fibers as in the proof of Theorem \ref{line} for $H^{n} \neq 3$, and \cite[Lemma 2.1]{Coskun2009} for $H^{n} = 3$. 
Thus the assertion holds.
\end{proof}
\begin{thm}\label{GMC}
Let $X$ be a del Pezzo manifold of Picard rank $1$ and dimension $n \ge 4$. 
Then for any $d \ge 1$, the unique component of $\Hom (\mathbb{P}^{1}, X, d)$ is a Manin component.
\end{thm}
\begin{proof}
By Lemma \ref{largea} and Lemma \ref{noa-cover}, there is no breaking morphism.
Thus any component $\Hom(\mathbb{P}^{1}, X, d)$ is a Manin component.
\end{proof}
\section*{Acknowledgments}
This paper is based on the author's master thesis \cite{Okamura2022}, which proved Theorem \ref{main} for general del Pezzo manifolds of degree $2$.
\par
I would like to thank my advisor Professor Sho Tanimoto for helpful discussions and continuous support.
I also would like to thank Professor Izzet Coskun, Professor Brian Lehmann and Professor Eric Riedl for helpful advices.
I am grateful to anonymous referees for their thoughtful comments and suggestions.
\par
The author was partially supported by JST FOREST program Grant number JPMJFR212Z and JSPS Bilateral Joint Research Projects Grant number JPJSBP120219935.
%


\begin{bibdiv}
\begin{biblist}

\bib{Abdelkerim2012}{article}{
      author={Abdelkerim, Richard},
      author={Coskun, Izzet},
       title={Parameter spaces of {S}chubert varieties in hyperplane sections of {G}rassmannians},
        date={2012},
        ISSN={0022-4049,1873-1376},
     journal={J. Pure Appl. Algebra},
      volume={216},
      number={4},
       pages={800\ndash 810},
         url={https://doi.org/10.1016/j.jpaa.2011.08.013},
      review={\MR{2864854}},
}

\bib{Batyrev1990}{article}{
      author={Batyrev, Victor~V.},
      author={Manin, Yuri~I.},
       title={Sur le nombre des points rationnels de hauteur born\'{e} des vari\'{e}t\'{e}s alg\'{e}briques},
        date={1990},
        ISSN={0025-5831,1432-1807},
     journal={Math. Ann.},
      volume={286},
      number={1-3},
       pages={27\ndash 43},
         url={https://doi.org/10.1007/BF01453564},
      review={\MR{1032922}},
}

\bib{Batyrev1998}{incollection}{
      author={Batyrev, Victor~V.},
      author={Tschinkel, Yuri},
       title={Tamagawa numbers of polarized algebraic varieties},
        date={1998},
       pages={299\ndash 340},
        note={Nombre et r\'{e}partition de points de hauteur born\'{e}e (Paris, 1996)},
      review={\MR{1679843}},
}

\bib{Beheshti2013}{article}{
      author={Beheshti, Roya},
      author={Kumar, N.~Mohan},
       title={Spaces of rational curves on complete intersections},
        date={2013},
        ISSN={0010-437X,1570-5846},
     journal={Compos. Math.},
      volume={149},
      number={6},
       pages={1041\ndash 1060},
         url={https://doi.org/10.1112/S0010437X12000504},
      review={\MR{3077661}},
}

\bib{Beheshti2022}{article}{
      author={Beheshti, Roya},
      author={Lehmann, Brian},
      author={Riedl, Eric},
      author={Tanimoto, Sho},
       title={Moduli spaces of rational curves on {F}ano threefolds},
        date={2022},
        ISSN={0001-8708,1090-2082},
     journal={Adv. Math.},
      volume={408},
       pages={Paper No. 108557, 60},
         url={https://doi.org/10.1016/j.aim.2022.108557},
      review={\MR{4456787}},
}

\bib{Beheshti2023}{article}{
      author={Beheshti, Roya},
      author={Lehmann, Brian},
      author={Riedl, Eric},
      author={Tanimoto, Sho},
       title={Rational curves on del {P}ezzo surfaces in positive characteristic},
        date={2023},
        ISSN={2330-0000},
     journal={Trans. Amer. Math. Soc. Ser. B},
      volume={10},
       pages={407\ndash 451},
         url={https://doi.org/10.1090/btran/138},
      review={\MR{4556219}},
}

\bib{Boucksom2013}{article}{
      author={Boucksom, S\'{e}bastien},
      author={Demailly, Jean-Pierre},
      author={P\u{a}un, Mihai},
      author={Peternell, Thomas},
       title={The pseudo-effective cone of a compact {K}\"{a}hler manifold and varieties of negative {K}odaira dimension},
        date={2013},
        ISSN={1056-3911,1534-7486},
     journal={J. Algebraic Geom.},
      volume={22},
      number={2},
       pages={201\ndash 248},
         url={https://doi.org/10.1090/S1056-3911-2012-00574-8},
      review={\MR{3019449}},
}

\bib{Bourqui2012}{article}{
      author={Bourqui, David},
       title={Moduli spaces of curves and {C}ox rings},
        date={2012},
        ISSN={0026-2285,1945-2365},
     journal={Michigan Math. J.},
      volume={61},
      number={3},
       pages={593\ndash 613},
         url={https://doi.org/10.1307/mmj/1347040261},
      review={\MR{2975264}},
}

\bib{Bourqui2016}{article}{
      author={Bourqui, David},
       title={Algebraic points, non-anticanonical heights and the {S}everi problem on toric varieties},
        date={2016},
        ISSN={0024-6115,1460-244X},
     journal={Proc. Lond. Math. Soc. (3)},
      volume={113},
      number={4},
       pages={474\ndash 514},
         url={https://doi.org/10.1112/plms/pdw035},
      review={\MR{3556489}},
}

\bib{Browning2017}{article}{
      author={Browning, Tim},
      author={Vishe, Pankaj},
       title={Rational curves on smooth hypersurfaces of low degree},
        date={2017},
        ISSN={1937-0652,1944-7833},
     journal={Algebra Number Theory},
      volume={11},
      number={7},
       pages={1657\ndash 1675},
         url={https://doi.org/10.2140/ant.2017.11.1657},
      review={\MR{3697151}},
}

\bib{Burke2022}{article}{
      author={Burke, Andrew},
      author={Jovinelly, Eric},
       title={Geometric {M}anin's {C}onjecture for {F}ano 3-folds},
        date={2022},
      eprint={arXiv:2209.05517v1},
}

\bib{Chen2020}{book}{
      author={Chen, Huayi},
      author={Moriwaki, Atsushi},
       title={Arakelov geometry over adelic curves},
      series={Lecture Notes in Mathematics},
   publisher={Springer, Singapore},
        date={2020},
      volume={2258},
        ISBN={978-981-15-1727-3; 978-981-15-1728-0},
         url={https://doi.org/10.1007/978-981-15-1728-0},
      review={\MR{4292529}},
}

\bib{Chung2018}{article}{
      author={Chung, Kiryong},
      author={Hong, Jaehyun},
      author={Lee, SangHyeon},
       title={Geometry of moduli spaces of rational curves in linear sections of {G}rassmannian {$Gr(2,5)$}},
        date={2018},
        ISSN={0022-4049,1873-1376},
     journal={J. Pure Appl. Algebra},
      volume={222},
      number={4},
       pages={868\ndash 888},
         url={https://doi.org/10.1016/j.jpaa.2017.05.011},
      review={\MR{3720857}},
}

\bib{Coskun2009}{article}{
      author={Coskun, Izzet},
      author={Starr, Jason},
       title={Rational curves on smooth cubic hypersurfaces},
        date={2009},
        ISSN={1073-7928,1687-0247},
     journal={Int. Math. Res. Not. IMRN},
      number={24},
       pages={4626\ndash 4641},
         url={https://doi.org/10.1093/imrn/rnp102},
      review={\MR{2564370}},
}

\bib{Franke1989}{article}{
      author={Franke, Jens},
      author={Manin, Yuri~I.},
      author={Tschinkel, Yuri},
       title={Rational points of bounded height on {F}ano varieties},
        date={1989},
        ISSN={0020-9910,1432-1297},
     journal={Invent. Math.},
      volume={95},
      number={2},
       pages={421\ndash 435},
         url={https://doi.org/10.1007/BF01393904},
      review={\MR{974910}},
}

\bib{Fulton1997}{incollection}{
      author={Fulton, W.},
      author={Pandharipande, R.},
       title={Notes on stable maps and quantum cohomology},
        date={1997},
   booktitle={Algebraic geometry---{S}anta {C}ruz 1995},
      series={Proc. Sympos. Pure Math.},
      volume={62, Part 2},
   publisher={Amer. Math. Soc., Providence, RI},
       pages={45\ndash 96},
         url={https://doi.org/10.1090/pspum/062.2/1492534},
      review={\MR{1492534}},
}

\bib{Harris2004}{article}{
      author={Harris, Joe},
      author={Roth, Mike},
      author={Starr, Jason},
       title={Rational curves on hypersurfaces of low degree},
        date={2004},
        ISSN={0075-4102,1435-5345},
     journal={J. Reine Angew. Math.},
      volume={571},
       pages={73\ndash 106},
         url={https://doi.org/10.1515/crll.2004.045},
      review={\MR{2070144}},
}

\bib{Hassett2015}{article}{
      author={Hassett, Brendan},
      author={Tanimoto, Sho},
      author={Tschinkel, Yuri},
       title={Balanced line bundles and equivariant compactifications of homogeneous spaces},
        date={2015},
        ISSN={1073-7928,1687-0247},
     journal={Int. Math. Res. Not. IMRN},
      number={15},
       pages={6375\ndash 6410},
         url={https://doi.org/10.1093/imrn/rnu129},
      review={\MR{3384482}},
}

\bib{Iskovskikh1999}{incollection}{
      author={Iskovskikh, V.~A.},
      author={Prokhorov, Yu.~G.},
       title={Fano varieties},
        date={1999},
   booktitle={Algebraic geometry, {V}},
      series={Encyclopaedia Math. Sci.},
      volume={47},
   publisher={Springer, Berlin},
       pages={1\ndash 247},
      review={\MR{1668579}},
}

\bib{Kim2001}{incollection}{
      author={Kim, B.},
      author={Pandharipande, R.},
       title={The connectedness of the moduli space of maps to homogeneous spaces},
        date={2001},
   booktitle={Symplectic geometry and mirror symmetry ({S}eoul, 2000)},
   publisher={World Sci. Publ., River Edge, NJ},
       pages={187\ndash 201},
         url={https://doi.org/10.1142/9789812799821_0006},
      review={\MR{1882330}},
}

\bib{Kollar1996}{book}{
      author={Koll\'{a}r, J\'{a}nos},
       title={Rational curves on algebraic varieties},
      series={Ergebnisse der Mathematik und ihrer Grenzgebiete. 3. Folge. A Series of Modern Surveys in Mathematics [Results in Mathematics and Related Areas. 3rd Series. A Series of Modern Surveys in Mathematics]},
   publisher={Springer-Verlag, Berlin},
        date={1996},
      volume={32},
        ISBN={3-540-60168-6},
         url={https://doi.org/10.1007/978-3-662-03276-3},
      review={\MR{1440180}},
}

\bib{Lazarsfeld2004}{book}{
      author={Lazarsfeld, Robert},
       title={Positivity in algebraic geometry. {I}},
      series={Ergebnisse der Mathematik und ihrer Grenzgebiete. 3. Folge. A Series of Modern Surveys in Mathematics [Results in Mathematics and Related Areas. 3rd Series. A Series of Modern Surveys in Mathematics]},
   publisher={Springer-Verlag, Berlin},
        date={2004},
      volume={48},
        ISBN={3-540-22533-1},
         url={https://doi.org/10.1007/978-3-642-18808-4},
        note={Classical setting: line bundles and linear series},
      review={\MR{2095471}},
}

\bib{Lehmann2022}{article}{
      author={Lehmann, Brian},
      author={Sengupta, Akash~Kumar},
      author={Tanimoto, Sho},
       title={Geometric consistency of {M}anin's conjecture},
        date={2022},
        ISSN={0010-437X,1570-5846},
     journal={Compos. Math.},
      volume={158},
      number={6},
       pages={1375\ndash 1427},
         url={https://doi.org/10.1112/s0010437x22007588},
      review={\MR{4472281}},
}

\bib{Lehmann2017}{article}{
      author={Lehmann, Brian},
      author={Tanimoto, Sho},
       title={On the geometry of thin exceptional sets in {M}anin's conjecture},
        date={2017},
        ISSN={0012-7094,1547-7398},
     journal={Duke Math. J.},
      volume={166},
      number={15},
       pages={2815\ndash 2869},
         url={https://doi.org/10.1215/00127094-2017-0011},
      review={\MR{3712166}},
}

\bib{Lehmann2019}{article}{
      author={Lehmann, Brian},
      author={Tanimoto, Sho},
       title={Geometric {M}anin's conjecture and rational curves},
        date={2019},
        ISSN={0010-437X,1570-5846},
     journal={Compos. Math.},
      volume={155},
      number={5},
       pages={833\ndash 862},
         url={https://doi.org/10.1112/s0010437x19007103},
      review={\MR{3937701}},
}

\bib{Lehmann2019a}{article}{
      author={Lehmann, Brian},
      author={Tanimoto, Sho},
       title={On exceptional sets in {M}anin's conjecture},
        date={2019},
        ISSN={2522-0144,2197-9847},
     journal={Res. Math. Sci.},
      volume={6},
      number={1},
       pages={Paper No. 12, 41},
         url={https://doi.org/10.1007/s40687-018-0174-9},
      review={\MR{3893027}},
}

\bib{Lehmann2021a}{article}{
      author={Lehmann, Brian},
      author={Tanimoto, Sho},
       title={Rational curves on prime {F}ano threefolds of index 1},
        date={2021},
        ISSN={1056-3911,1534-7486},
     journal={J. Algebraic Geom.},
      volume={30},
      number={1},
       pages={151\ndash 188},
         url={https://doi.org/10.1090/jag/751},
      review={\MR{4233180}},
}

\bib{Lehmann2022a}{article}{
      author={Lehmann, Brian},
      author={Tanimoto, Sho},
       title={Classifying sections of del {P}ezzo fibrations, {II}},
        date={2022},
        ISSN={1465-3060,1364-0380},
     journal={Geom. Topol.},
      volume={26},
      number={6},
       pages={2565\ndash 2647},
         url={https://doi.org/10.2140/gt.2022.26.2565},
      review={\MR{4521249}},
}

\bib{Lehmann2023}{article}{
      author={Lehmann, Brian},
      author={Tanimoto, Sho},
       title={Classifying sections of del {P}ezzo fibrations, {I}},
        date={2023},
     journal={J. Eur. Math. Soc. (JEMS)},
}

\bib{Lehmann2018}{article}{
      author={Lehmann, Brian},
      author={Tanimoto, Sho},
      author={Tschinkel, Yuri},
       title={Balanced line bundles on {F}ano varieties},
        date={2018},
        ISSN={0075-4102,1435-5345},
     journal={J. Reine Angew. Math.},
      volume={743},
       pages={91\ndash 131},
         url={https://doi.org/10.1515/crelle-2015-0084},
      review={\MR{3859270}},
}

\bib{Mori1982}{article}{
      author={Mori, Shigefumi},
       title={Threefolds whose canonical bundles are not numerically effective},
        date={1982},
        ISSN={0003-486X},
     journal={Ann. of Math. (2)},
      volume={116},
      number={1},
       pages={133\ndash 176},
         url={https://doi.org/10.2307/2007050},
      review={\MR{662120}},
}

\bib{Okamura2022}{article}{
      author={Okamura, Fumiya},
       title={Rational curves on del {P}ezzo manifolds of dimension at least $4$},
        date={2022},
     journal={Master's Thesis, Kumamoto University},
}

\bib{Peyre1995}{article}{
      author={Peyre, Emmanuel},
       title={Hauteurs et mesures de {T}amagawa sur les vari\'{e}t\'{e}s de {F}ano},
        date={1995},
        ISSN={0012-7094,1547-7398},
     journal={Duke Math. J.},
      volume={79},
      number={1},
       pages={101\ndash 218},
         url={https://doi.org/10.1215/S0012-7094-95-07904-6},
      review={\MR{1340296}},
}

\bib{Piontkowski1999}{article}{
      author={Piontkowski, J.},
      author={Van~de Ven, A.},
       title={The automorphism group of linear sections of the {G}rassmannians {${\bf G}(1,N)$}},
        date={1999},
        ISSN={1431-0635,1431-0643},
     journal={Doc. Math.},
      volume={4},
       pages={623\ndash 664},
      review={\MR{1719726}},
}

\bib{Riedl2019}{article}{
      author={Riedl, Eric},
      author={Yang, David},
       title={Kontsevich spaces of rational curves on {F}ano hypersurfaces},
        date={2019},
        ISSN={0075-4102,1435-5345},
     journal={J. Reine Angew. Math.},
      volume={748},
       pages={207\ndash 225},
         url={https://doi.org/10.1515/crelle-2016-0027},
      review={\MR{3918434}},
}

\bib{Shimizu2022}{article}{
      author={Shimizu, Nobuki},
      author={Tanimoto, Sho},
       title={The spaces of rational curves on del {P}ezzo threefolds of degree one},
        date={2022},
        ISSN={2199-675X,2199-6768},
     journal={Eur. J. Math.},
      volume={8},
      number={1},
       pages={291\ndash 308},
         url={https://doi.org/10.1007/s40879-021-00516-2},
      review={\MR{4389494}},
}

\bib{Takagi2012}{article}{
      author={Takagi, Hiromichi},
      author={Zucconi, Francesco},
       title={Geometries of lines and conics on the quintic del {P}ezzo 3-fold and its application to varieties of power sums},
        date={2012},
        ISSN={0026-2285,1945-2365},
     journal={Michigan Math. J.},
      volume={61},
      number={1},
       pages={19\ndash 62},
         url={https://doi.org/10.1307/mmj/1331222846},
      review={\MR{2904000}},
}

\bib{Tanimoto2021}{article}{
      author={Tanimoto, Sho},
       title={An introduction to {G}eometric {M}anin's conjecture},
        date={2021},
     journal={Proceedings of Miyako-no-Seihoku Algebraic Geometry Symposium 2021, “Positivity of tangent bundles and related topics”},
       pages={102\ndash 118},
      eprint={arXiv:2110.06660},
}

\bib{Testa2005}{book}{
      author={Testa, Damiano},
       title={The {S}everi problem for rational curves on del {P}ezzo surfaces},
   publisher={ProQuest LLC, Ann Arbor, MI},
        date={2005},
         url={http://gateway.proquest.com/openurl?url_ver=Z39.88-2004&rft_val_fmt=info:ofi/fmt:kev:mtx:dissertation&res_dat=xri:pqdiss&rft_dat=xri:pqdiss:0808059},
        note={Thesis (Ph.D.)--Massachusetts Institute of Technology},
      review={\MR{2717265}},
}

\bib{Testa2009}{article}{
      author={Testa, Damiano},
       title={The irreducibility of the spaces of rational curves on del {P}ezzo surfaces},
        date={2009},
        ISSN={1056-3911,1534-7486},
     journal={J. Algebraic Geom.},
      volume={18},
      number={1},
       pages={37\ndash 61},
         url={https://doi.org/10.1090/S1056-3911-08-00484-0},
      review={\MR{2448278}},
}

\bib{Thomsen1998}{article}{
      author={Thomsen, Jesper~Funch},
       title={Irreducibility of {$\overline{M}_{0,n}(G/P,\beta)$}},
        date={1998},
        ISSN={0129-167X,1793-6519},
     journal={Internat. J. Math.},
      volume={9},
      number={3},
       pages={367\ndash 376},
         url={https://doi.org/10.1142/S0129167X98000154},
      review={\MR{1625369}},
}

\end{biblist}
\end{bibdiv}
\end{document}